\documentclass{article}
\usepackage[top=30truemm, bottom=30truemm, left=25truemm, right=25truemm]{geometry}

\usepackage{amsmath, amssymb, amsthm}
\usepackage{tikz}

\usepackage{url}
\usepackage{mathrsfs}

\newtheorem{defi}{Definition}[section]
\newtheorem{nota}[defi]{Notation}
\newtheorem{lem}[defi]{Lemma}
\newtheorem{prop}[defi]{Proposition}
\newtheorem{thm}[defi]{Theorem}
\newtheorem{cor}[defi]{Corollary}

\allowdisplaybreaks

\title{Terwilliger Algebra of the Ordered Hamming Scheme}
\author{Yuta Watanabe}
\date{July 9, 2024}

\begin{document}

\maketitle

\begin{abstract}

This paper delves into the Terwilliger algebra associated with the ordered Hamming scheme, which extends from the wreath product of one-class association schemes and was initially introduced by Delsarte as a natural expansion of the Hamming schemes. 
Levstein, Maldonado and Penazzi have shown that the Terwilliger algebra of the Hamming scheme of length $n$ is the $n$-fold symmetric tensor algebra of that of the one-class association scheme. Furthermore, Bhattacharyya, Song and Tanaka have established that the Terwilliger algebra of the wreath product of a one-class association scheme is a direct sum of the ``primary'' subalgebra and commutative subalgebras.
This paper extends these findings to encompass both conclusions.

\end{abstract}

\section{Introduction}

Commutative association schemes are known as a unifying concept of several fields such as coding theory, design theory, algebraic graph theory and finite group theory.
In 1992,
The Terwilliger algebra was first introduced in \cite{Terwilliger} as a new method of studying commutative association schemes.
The Terwilliger algebra has been studied on various classes of association schemes.
For instance, research has focused on Hamming schemes \cite{G, LMP}, Johnson schemes \cite{LM, LMW}, Grassmann schemes \cite{LIW}, Doob schemes \cite{Tanabe}, among others.
For details, we refer the readers to the references \cite{BBIT} and the references therein.
Recently, Bernard-Crampe-d'Andecy-Vinet-Zaimi \cite{BCDVZ} and Bannai-Kurihara-Zhao-Zhu \cite{BKZZ}
introduce multivariate $P$- and/or $Q$-polynomial association schemes.
As stated in \cite{BKZZ}, the ordered Hamming scheme, which is the main topic of this paper, is one such example.

In this paper, we focus on the Terwilliger algebra of the ordered Hamming scheme.
The ordered Hamming scheme is defined as the extension, or symmetrization, of the wreath product of one-class association schemes. This extension was introduced by Delsarte \cite[Section 2.5]{D} as a natural generalization of the Hamming schemes. Indeed, the extension of a one-class association scheme is specifically referred to as the (ordinary) Hamming scheme.
The ordered Hamming scheme has been appeared in the studies on ordered orthogonal arrays \cite{MS} and on quantum spin chain \cite{MTV}.
By Go \cite{G} (for the binary case) and Levstein-Maldonado-Penazzi \cite{LMP} (for the general case),
the Terwilliger algebra of the Hamming scheme of length $n$ is known to be the $n$-fold symmetric tensor algebra of the Terwilliger algebra of the one-class association scheme.
Additionally, 
by Bhattacharyya-Song-Tanaka \cite{BST},
the Terwilliger algebra of the wreath product of one-class association schemes is known to be 
the direct sum of the ``primary'' subalgebra and commutative subalgebras, which is called almost commutative by Tanaka \cite{Tanaka}.
In this paper, we extend these results to cover both conclusions.
The main results of this paper are Theorems \ref{main} and \ref{main2}.

For the sake of completeness, we also consider the Bose-Mesner algebra of the ordered Hamming scheme,
even though it has already been studied.
In this paper, we state the first and second eigenmatrices of the ordered Hamming scheme. Refer to Proposition \ref{vPQ}.
Delsarte \cite{D} showed that the entries of the first and second eigenmatrices of the Hamming scheme can be described using the Krawchouk polynomial, an orthogonal polynomial.
Mizukawa-Tanaka \cite{MT} extended this finding, proving that the entries of the first and second eigenmatrices of an ordered Hamming scheme can be expressed in terms of the multivariate Krawchouk polynomial.
We note that Mizukawa-Tanaka \cite{MT} considered a more general scenario from a group-theoretical standpoint.
We reprove these results from a combinatorial perspective.

We organize this paper as follows.
In Section 2, we review the basic notation for tensor products and present some useful lemmas.
In Section 3, we review the concept of association schemes and Bose-Mesner algebras.
In Section 4, we discuss Terwilliger algebras, define the ``primary'' subalgebras, and prove some useful lemmas.
In Section 5, we introduce the ordered Hamming schemes.
In Sections 6 and 7, we determine the structure of the Bose-Mesner algebra of the ordered Hamming schemes.
In Sections 8 and 9, we establish the structure of the Terwilliger algebra of the ordered Hamming schemes, including our main theorem.
In Section 10, we simplify our main theorem by restating it for small parameters. This result also includes the case for the (ordinary) Hamming scheme.

We end this section by deriving some notation about matrix algebras that are used
throughout this paper.
For a non-empty finite set $X$,
let $\operatorname{Mat}_X(\mathbb{C})$ denote the algebra over the complex field $\mathbb{C}$
consisting of matrices that have rows and columns indexed by $X$ and all entries in $\mathbb{C}$.
If $X$ is a set of $n$ elements, we abbreviate $\operatorname{Mat}_X(\mathbb{C})$ to $\operatorname{Mat}_n(\mathbb{C})$.
Let $I_X, J_X, O_X \in \operatorname{Mat}_X(\mathbb{C})$ denote the identity matrix, the all-ones matrix and the zero matrix, respectively.
For $A, B \in \operatorname{Mat}_X(\mathbb{C})$,
let ${}^t A$ and $\overline{A}$ denote the transpose and complex conjugate of $A$ and let $A \circ B$ denote the entry-wise product (or the Hadamard product) of $A$ and $B$.
For $A = [a_{i,j}]_{i,j=1}^n \in \operatorname{Mat}_n(\mathbb{C})$ and $B \in \operatorname{Mat}_m(\mathbb{C})$,
let $A \otimes B$ denote the tensor product (or Kronecker product) of $A$ and $B$
define by 
\[
A \otimes B = 
\begin{bmatrix}
a_{1,1}B & a_{1,2}B & \cdots & a_{1,n}B \\
a_{2,1}B & a_{2,2}B & \cdots & a_{2,n}B \\
\vdots & \vdots & \ddots & \vdots \\
a_{n,1}B & a_{n,2}B & \cdots & a_{n,n}B
\end{bmatrix}
\in \operatorname{Mat}_{nm}(\mathbb{C}).
\]

\section{Tensor products}
In this section we derive some notation about tensor products that are used in the main body of the paper.
Let $V$ be a vector space over $\mathbb{C}$.
For a positive integer $n$, let
\[
V^{\otimes n} = \underbrace{V \otimes V \otimes \cdots \otimes V}_{\text{$n$ factors}}
\]
denote the $n$-fold tensor product of $V$ with itself, and 
\begin{align*}
v^{\otimes n} = \underbrace{v \otimes v \otimes \cdots \otimes v}_{\text{$n$ factors}}
&&
v \in V.
\end{align*}
We define the linear transformation of $V^{\otimes n}$ by
\begin{align*}
\mathscr{S}_n \left( v_1 \otimes v_2 \otimes \cdots \otimes v_n\right)
=\frac{1}{n!} \sum_{\pi \in \mathfrak{S}_n} v_{\pi(1)} \otimes v_{\pi(2)} \otimes \cdots \otimes v_{\pi(n)}
&&
v_1, v_2, \ldots, v_n \in V,
\end{align*}
where $\mathfrak{S}_n$ denote the symmetric group of order $n$.
We call $\mathscr{S}_n$ the symmetrizer of $V^{\otimes n}$.
A tensor product $v \in V^{\otimes n}$ is called a symmetric tensor if 
$\mathscr{S}_n(v) = v$.
Let $\operatorname{Sym}^n(V)$ denote the linear subspace containing of all symmetric tensors in $V^{\otimes n}$.
We call $\operatorname{Sym}^n(V)$ the $n$-fold symmetric tensor space.
For convenience, we define $\operatorname{Sym}^0(V) = \mathbb{C}$.
The following notation is used in Levstein-Maldonado-Penazzi \cite{LMP}.
For $v_1, v_2, \ldots, v_m \in V$ and 
non-negative integers $i_1, i_2, \ldots, i_m$ with $i_1 + i_2 + \cdots + i_m = n$,
we define
\[
\mathscr{L}_{i_1, i_2, \ldots, i_m}(v_1, v_2, \ldots, v_m)
= 
\binom{n}{i_1, i_2, \ldots, i_m}
\mathscr{S}_n \left( v_1^{\otimes i_1} \otimes v_2^{\otimes i_2} \otimes \cdots \otimes v_m^{\otimes i_m}\right).
\]
This means $\mathscr{L}_{i_1, i_2, \ldots, i_m}(v_1, v_2, \ldots, v_m)$ is the sum of symmetric tensors each of which has exactly $i_t$ numbers of $v_t$ for $1 \le t \le m$ and each combination appears exactly once.
For example,
$\mathscr{L}_{1,2}(v, w) = v \otimes w \otimes w + w \otimes v \otimes w +  w \otimes w \otimes v$.

\begin{prop}[{\cite[Proposition 2.6]{Y}}]\label{sym}
For a positive integer $n$ and for a vector space $V$ over $\mathbb{C}$ with basis $\{v_1, v_2, \ldots, v_m\}$,
the $n$-fold symmetric tensor space $\operatorname{Sym}^n(V)$ has a basis 
\[
\left\{
\mathscr{L}_{i_1, i_2, \ldots, i_m}(v_1, v_2, \ldots, v_m)
\mid i_1, i_2, \ldots, i_m \in \{0,1,\ldots, n\}, i_1+i_2+\cdots+i_m = n
\right\}.
\]
The dimension of $\operatorname{Sym}^n(V)$ is given by
\[
\dim \left( \operatorname{Sym}^n(V)\right) = \binom{n+m-1}{n}.
\]
\end{prop}

Let $n,m$ be non-negative integers.
For two symmetric tensors $u \in \operatorname{Sym}^n(V)$ and 
$w \in \operatorname{Sym}^m(V)$,
the symmetric product of $u$ and $w$ is defined by
\[
u \odot w = \binom{n+m}{n}\mathscr{S}_{n+m}\left(u \otimes w\right) \in \operatorname{Sym}^{n+m}(V).
\]
For convenience, we identify $V = V \otimes \mathbb{C} = \mathbb{C} \otimes V$ by $\alpha v = v \otimes \alpha = \alpha \otimes v$ for $v \in V$ and $\alpha \in \mathbb{C}$.
By this identification, we have $\alpha u = u \odot \alpha = \alpha \odot u$ for $u \in \operatorname{Sym}^n(V)$ and $\alpha \in \mathbb{C}$.
Observe that
\[
\mathscr{L}_{i_1, i_2, \ldots, i_k}(v_1, v_2, \ldots, v_k) \odot
\mathscr{L}_{j_1, j_2, \ldots, j_l}(w_1,w_2,\ldots,w_l) 
=
\mathscr{L}_{i_1, i_2, \ldots, i_k,j_1, j_2, \ldots, j_l}(v_1, v_2, \ldots, v_k,w_1,w_2,\ldots,w_l),
\]
for $v_1, v_2, \ldots, v_k \in V$, $w_1,w_2,\ldots,w_l \in W$ and 
non-negative integers $i_1, i_2, \ldots, i_k, j_1, j_2, \ldots, j_l$ with $i_1 + i_2 + \cdots + i_k = n$, $j_1 + j_2 + \cdots + j_l = m$.
This is the reason why we put the coefficient $\binom{n+m}{n}$ in the definition of the symmetric product.
For subspaces $U \subset \operatorname{Sym}^n(V)$ 
and $W \subset \operatorname{Sym}^m(V)$, 
the symmetric product of $U$ and $W$, denoted by $U \odot W$,
is the vector space spanned by $u \odot w$ for $u \in U$, $w \in W$.
By the identification $V = V \otimes \mathbb{C} = \mathbb{C} \otimes V$, we identify $W \odot \mathbb{C} = \mathbb{C} \odot W = W$
for a subspace $W \subset \operatorname{Sym}^n(V)$.

\begin{lem}\label{sym:V+W}
For a positive integer $n$ and vector spaces $V, W$ over $\mathbb{C}$, we have
\[
\operatorname{Sym}^n(V \oplus W) =
\bigoplus_{d=0}^n \operatorname{Sym}^{n-d}(V) \odot \operatorname{Sym}^d(W).
\]
\end{lem}
\begin{proof}
Routine using Proposition \ref{sym} and the definition of the symmetric product.
\end{proof}

Let $\mathcal{A}$ and $\mathcal{B}$ be associative algebras.
We make the vector space $\mathcal{A} \otimes \mathcal{B}$ into an associative algebra relative to the product
\begin{align*}
(A \otimes B)(A' \otimes B') = (AA') \otimes (BB')
&&
A,A' \in \mathcal{A}, \
B,B' \in \mathcal{B}.
\end{align*}
For s positive integer $n$ and an associative algebra $\mathcal{A}$,
the $n$-fold symmetric tensor space $\operatorname{Sym}^n(\mathcal{A})$
become an associative algebra.
We call it the $n$-fold symmetric tensor algebra.
If $\mathcal{A}$ is unital, then so is $\operatorname{Sym}^n(\mathcal{A})$.

\begin{lem}\label{sym:basis}
Let $n$ be a positive integer and 
let $\mathcal{A}$ be a unital associative algebra.
The basis for the $n$-fold symmetric tensor space $\operatorname{Sym}^n(\mathcal{A})$ in Proposition \ref{sym} with respect to the primitive idempotent basis for $\mathcal{A}$
is also the primitive idempotent basis.
\end{lem}
\begin{proof}
Routine.
\end{proof}

\begin{lem}\label{sym:gen}
Let $n$ be a positive integer and 
let $\mathcal{A}$ be a unital associative algebra with its identity element $I$.
Let $\{A_1, A_2, \ldots, A_m\}$ denote the primitive idempotent basis for $\mathcal{A}$.
Then the $n$-fold symmetric tensor algebra $\operatorname{Sym}^n(\mathcal{A})$ is generated by
$\{\mathscr{L}_{1,n-1}(A_j,I) \mid 1 \le j \le m\}$.
\end{lem}
\begin{proof}
Let $\mathcal{S}$ denote the associative algebra generated by
$\{\mathscr{L}_{1,n-1}(A_j,I) \mid 1 \le j \le m\}$.
Since these generators are symmetric tensors, $\mathcal{S}$ is a subalgebra of $\operatorname{Sym}^n(\mathcal{A})$.
It suffices to show  $\operatorname{Sym}^n(\mathcal{A}) \subset \mathcal{S}$.

Fix $1 \le j \le m$.
We first show $\mathscr{L}_{i,n-i}(A_j,I) \in \mathcal{S}$ for $1 \le i \le n$ by induction on $i$.
If $i=1$, then it is one of the generators and hence in $\mathcal{S}$.
For $i \ge 2$,
since $A_j^2 = A_j$, we have
\[
\mathscr{L}_{i-1,n-i+1}(A_j,I)\mathscr{L}_{1,n-1}(A_j,I)
= (i-1) \mathscr{L}_{i-1,n-i+1}(A_j,I) + i \mathscr{L}_{i,n-i}(A_j,I).
\]
By the inductive hypothesis, we get $\mathscr{L}_{i,n-i}(A_j,I) \in \mathcal{S}$.
Therefore, the claim is true for $1 \le i \le n$.

Fix non-negative integers $i_1,i_2,\ldots,i_m$ with $i_1+i_2+\cdots+i_m = n$.
We next show 
\[
\mathscr{L}_{i_1,i_2, \ldots, i_t, j_t}(A_1,A_2, \ldots, A_t, I) \in \mathcal{S}
\]
for $1 \le t \le m$, where $j_t = n-(i_1+i_2+\cdots+i_t)$ by induction on $t$.
We have already shown the case $t=1$.
For $t \ge 2$, since $A_iA_t = A_tA_i = 0$ for $1 \le i \le t-1$, we have
\[
\mathscr{L}_{i_1,i_2, \ldots, i_{t-1}, j_{t-1}}(A_1,A_2, \ldots, A_{t-1}, I)
\mathscr{L}_{i_t,n-i_t}(A_t,I)
=
\mathscr{L}_{i_1, i_2, \ldots, i_t, j_t}(A_1, A_2, \ldots, A_t, I).
\]
By the inductive hypothesis, we get $\mathscr{L}_{i_1, i_2, \ldots, i_t, j_t}(A_1, A_2, \ldots, A_t, I) \in \mathcal{S}$.
Thus the claim is true for $1 \le t \le m$.
Specifically, by setting $t=m$, 
we have
$\mathscr{L}_{i_1, i_2, \ldots, i_m}(A_1, A_2, \ldots, A_m) \in \mathcal{S}$
for any non-negative integers $i_1,i_2,\ldots,i_m$ with $i_1+i_2+\cdots+i_m = n$.
By Proposition \ref{sym},
a basis for $\operatorname{Sym}^n(\mathcal{A})$ is entirely contained in $\mathcal{S}$.
This means $\operatorname{Sym}^n(\mathcal{A}) \subset \mathcal{S}$.
\end{proof}

\begin{lem}\label{sym:CdotD}
Let $\mathcal{A}$ denote an associative algebra.
For positive integers $n,m$ and 
for subalgebras $\mathcal{B}, \mathcal{C} \subset \mathcal{A}$,
if $BC = CB = 0$ for $B \in \mathcal{B}$, $C \in \mathcal{C}$,
we have
\begin{align*}
(U \odot W)(U' \odot W') = (UU') \odot (WW')
&&
U, U' \in \operatorname{Sym}^n(\mathcal{B}), \
W, W' \in \operatorname{Sym}^m(\mathcal{C}).
\end{align*}
In particular, $\operatorname{Sym}^n(\mathcal{B}) \odot \operatorname{Sym}^m(\mathcal{C})$ is isomorphic, as associative algebras, to $\operatorname{Sym}^n(\mathcal{B}) \otimes \operatorname{Sym}^m(\mathcal{C})$.
\end{lem}
\begin{proof}
Routine using Proposition \ref{sym} and the definition of the symmetric product.
\end{proof}

\section{Association schemes}\label{AS}
In this section, we briefly recall the notation and some basic facts about association schemes.
For more information, we refer the reader to \cite{BBIT} and the references therein.
A pair $\mathfrak{X} = (X,\{R_i\}_{i=0}^d)$ of a finite set $X$ with at least two elements and a set of subsets in $X \times X$ is called a $d$-class commutative association scheme if it satisfies the following five conditions:
\begin{enumerate}
\item[(R1)] $R_0 = \{(x,x) \mid x \in X\}$;
\item[(R2)] $\{R_i\}_{i=0}^d$ is a partition of $X \times X$;
\item[(R3)] For $0 \le i \le d$, there exists $0 \le j \le d$ such that $R_j = \{(y,x) \mid (x,y) \in R_i\}$;
\item[(R4)] For $0 \le i,j,k \le d$, and a given $(x,y) \in R_k$, $p_{i,j}^k = |\{z \in X \mid (x,z) \in R_i, (z,y) \in R_j\}|$ depends only on $i,j,k$;
\item[(R5)] For $0 \le i,j,k \le d$, the numbers given above satisfy $p_{i,j}^k = p_{j,i}^k$.
\end{enumerate}
The association scheme is called symmetric if $i=j$ for all $0 \le i \le d$ in the condition (R3).

Let $\mathfrak{X} = (X,\{R_i\}_{i=0}^d)$ be a commutative association scheme.
For $0 \le i \le d$, define a matrix $A_i \in \operatorname{Mat}_X(\mathbb{C})$ by
\begin{align*}
(A_i)_{x,y} =
\begin{cases}
1 & \text{if $(x,y) \in R_i$}, \\ 0 & \text{if $(x,y) \not\in R_i$}
\end{cases}
&&
x,y \in X.
\end{align*}
We call $A_i$ the $i$-th adjacency matrix of $\mathfrak{X}$.
If the association scheme is symmetric, adjacency matrices are symmetric.
For $0 \le i \le d$, each $A_i$ has a constant row sum, say $k_i$, and we call it the valency of $\mathfrak{X}$.
The matrices $\{A_i\}_{i=0}^d$ form a basis for a commutative semi-simple subalgebra $\mathcal{A}$ of $\operatorname{Mat}_X(\mathbb{C})$,
known as the Bose-Mesner algebra of $\mathfrak{X}$.
The Bose-Mesner algebra $\mathcal{A}$ possesses a second basis comprising $\{E_i\}_{i=0}^d$, which satisfies the following four conditions:
\begin{enumerate}
\item[(E1)] $|X|E_0 = J_X$;
\item[(E2)] $E_0 + E_1 + \cdots + E_d = I_X$;
\item[(E3)] For $0 \le i \le d$, there exists $0 \le j \le d$ such that ${}^t E_i = \overline{E}_i = E_j$;
\item[(E4)] For $0 \le i,j \le d$, $E_iE_j = \delta_{i,j}E_i$.
\end{enumerate}
Furthermore, it is important to note that this basis is unique up to permutations of $\{E_i\}_{i=1}^d$.
These matrices $\{E_i\}_{i=0}^d$ are referred to as the primitive idempotents of $\mathfrak{X}$.
For $0 \le i \le d$, the trace $m_i = \operatorname{trace}(E_i)$ is called the multiplicity of $\mathfrak{X}$. 
Note that if the association scheme is symmetric, $j = i$ in the condition (E3).

Since the adjacency matrices and the primitive idempotents are both basis for the Bose-Mesner algebra,
there exist $P_i(j), Q_i(j) \in \mathbb{C}$ such that
\begin{align*}
A_i = \sum_{j=0}^d P_i(j) E_j,
&&
E_i = |X|^{-1}\sum_{j=0}^d Q_i(j) A_j
\end{align*}
for $0 \le i \le d$.
By (E4) and (R2), we have
\begin{align*}
A_iE_j = P_i(j)E_j,
&&
E_i \circ A_j = Q_i(j)A_j
\end{align*}
for $0 \le i,j \le d$,
where $\circ$ denote the entry-wise product (or the Hadamard product).
The base change matrices $P = [P_j(i)]_{i,j=0}^d$ and $Q = [Q_j(i)]_{i,j=0}^d$,
whose $(i,j)$-entry is $P_j(i)$, $Q_j(i)$, respectively, are called the first eigenmatrix and the second eigenmatrix of the association scheme $\mathfrak{X}$, respectively.
An association scheme is called self-dual whenever its eigenmatrices satisfy $P = \overline{Q}$.

\section{Terwilliger algebras}

Let $\mathfrak{X}$ denote the $d$-class commutative association scheme on a finite set $X$.
We continue to use the notation from Section \ref{AS}.
Fix $x \in X$. For $0 \le i \le d$,
we define $E^*_i = E^*_i(x)$ as the diagonal matrix in $\operatorname{Mat}_X(\mathbb{C})$ by
\begin{align*}
(E^*_i)_{y,y} = (A_i)_{x,y}
&&
y \in X.
\end{align*}
We refer to $\{E^*_i\}_{i=0}^d$ as the dual primitive idempotents of $\mathfrak{X}$.
They satisfy the following.
\begin{enumerate}
\item[(E1*)] $E_0^*$ is the $(x,x)$ matrix unit;
\item[(E2*)] $E_0^* + E_1^* + \cdots + E_d^* = I_X$;
\item[(E3*)] For $0 \le i \le d$, ${}^t E_i^* = \overline{E}_i^* = E_i^*$;
\item[(E4*)] For $0 \le i,j \le d$, $E_i^*E_j^* = \delta_{i,j}E_i^*$.
\end{enumerate}
The linear subspace $\mathcal{A}^* = \mathcal{A}^*(x)$ of $\operatorname{Mat}_X(\mathbb{C})$ spanned by $\{E^*_i\}_{i=0}^d$ becomes commutative semi-simple subalgebra.
This subalgebra is known as the dual Bose-Mesner algebra of $\mathfrak{X}$.
The Terwilliger algebra $\mathcal{T} = \mathcal{T}(x)$ of $\mathfrak{X}$ is the subalgebra in $\operatorname{Mat}_X(\mathbb{C})$ generated by both the Bose-Mesner algebra $\mathcal{A}$ and the dual Bose-Mesner algebra $\mathcal{A}^*$.

The following lemma is useful.
\begin{lem}\label{lem}
With above notation, for $0 \le i \le d$, we have the following.
\begin{enumerate}
\item $E_0E_i^*E_0 = |X|^{-1}k_iE_0$.
\item $E_0^*E_iE_0^* = |X|^{-1}m_iE_0^*$.
\item $E_0E_i^* = E_0E_0^*A_i$.
\end{enumerate}
\end{lem}
\begin{proof}
For $y, z \in X$,
it suffices to verify that the $(y, z)$-entries on both sides match.
(i) The $(y, z)$-entry matches at $|X|^{-2} k_i$.
(ii) The $(y, z)$-entry matches at $|X|^{-1}m_i$ when $x = y = z$, and at $0$ otherwise.
(iii) The $(y, z)$-entry matches at $|X|^{-1}$ when $(x,z) \in R_i$, and at $0$ otherwise.
\end{proof}

Let $\mathcal{T}_0$ be the linear subspace of $\mathcal{T}$ spanned by $E_i^*E_0E_j^*$ $(0 \le i,j \le d)$.
By (E4*), the set of matrices $\{E_i^*E_0E_j^* \mid 0 \le i,j \le d\}$ forms a basis for $\mathcal{T}_0$, which we call the standard basis for $\mathcal{T}_0$.
By Lemma \ref{lem} (i), $\mathcal{T}_0$ is closed under matrix multiplication.
Moreover, there exists a $\mathbb{C}$-algebra isomorphism $\varphi: \mathcal{T}_0 \to \operatorname{Mat}_{d+1}(\mathbb{C})$ such that
$\varphi(E_i^*E_0E_j^*) = |X|^{-1}\sqrt{k_ik_j} M_{i,j}$,
where $M_{i,j}$ denote the $(i,j)$ matrix unit in $\operatorname{Mat}_{d+1}(\mathbb{C})$.
The subalgebra $\mathcal{T}_0$ is called the primary subalgebra of $\mathcal{T}$.
Dually, let $\mathcal{T}_0'$ be the linear subspace of $\mathcal{T}$ spanned by $E_iE_0^*E_j$ $(0 \le i,j \le d)$.
By (E4), the set of matrices $\{E_iE_0^*E_j \mid 0 \le i,j \le d\}$ forms a basis for $\mathcal{T}_0'$.
By Lemma \ref{lem} (ii), $\mathcal{T}_0'$ is closed under matrix multiplication.
Moreover, there exists a $\mathbb{C}$-algebra isomorphism $\varphi^*: \mathcal{T}_0 \to \operatorname{Mat}_{d+1}(\mathbb{C})$ such that
$\varphi^*(E_iE_0^*E_j) = |X|^{-1}\sqrt{m_im_j} M_{i,j}$,
where $M_{i,j}$ denote the $(i,j)$ matrix unit in $\operatorname{Mat}_{d+1}(\mathbb{C})$.
By Lemma \ref{lem} (iii) and its transpose,
we have $\mathcal{T}_0 = \mathcal{T}_0'$.
Therefore, we call the basis $\{E_iE_0^*E_j \mid 0 \le i,j \le d\}$ the dual standard basis for $\mathcal{T}_0$.
By these observation, we have the following result.

\begin{prop}\label{T0}
The primary subalagebra of the Terwilliger algebra has two bases,
the standard basis $\{E_iE_0^*E_j \mid 0 \le i,j \le d\}$
and 
the dual standard basis $\{E_iE_0^*E_j \mid 0 \le i,j \le d\}$.
With respect to these two bases, there are two $\mathbb{C}$-algebra isomorphisms $\varphi, \varphi^*: \mathcal{T}_0 \to \operatorname{Mat}_{d+1}(\mathbb{C})$ such that
\begin{align*}
\varphi(E_i^*E_0E_j^*) = |X|^{-1}\sqrt{k_ik_j} M_{i,j},
&&
\varphi^*(E_iE_0^*E_j) = |X|^{-1}\sqrt{m_im_j} M_{i,j},
\end{align*}
where $M_{i,j}$ denote the $(i,j)$ matrix unit in $\operatorname{Mat}_{d+1}(\mathbb{C})$.
\end{prop}

For $0 \le i \le d$, we define the following matrices in the primary subalgebra $\mathcal{T}_0$:
\begin{align*}
F_i = \frac{|X|}{m_i} E_iE_0^*E_i,
&&
F_i^* = \frac{|X|}{k_i}E_i^*E_0E_i^*.
\end{align*}
These are the preimages of the $(i,i)$ matrix unit
under the isomorphisms $\varphi, \varphi^*$ in Proposition \ref{T0}, respectively.
In addition, let $F^\natural$ denote the identity element in $\mathcal{T}_0$,
which is equivalently defined as the common preimage of the identity matrix under the isomorphisms $\varphi, \varphi^*$ in Proposition \ref{T0}.
The following lemma is easily obtained.
\begin{lem}\label{lem2}
With above notation, for $0 \le i,j \le d$, we have the following.
\begin{enumerate}
\item $F_0 = E_0$.
\item $F_0^* = E_0^*$.
\item $F_i^*F_0F_j^* = E_i^*E_0E_j^*$.
\item $F_iF_0^*F_j = E_iE_0^*E_j$.
\item $F_iF_j = \delta_{i,j}F_i$.
\item $F_i^*F_j^* = \delta_{i,j}F_i^*$.
\item $F^\natural = F_0 + F_1 + \cdots + F_d = F_0^* + F_1^* + \cdots + F_d^*$.
\end{enumerate}
\end{lem}
\begin{proof}
(i) and (iii) follow from Lemma \ref{lem} (i).
(ii) and (iv) follow from Lemma \ref{lem} (ii).
(v), (vi) and (vii) follow from Proposition \ref{T0}.
\end{proof}

\begin{lem}\label{lem:eqF}
With above notation, the following are equivalent.
\begin{enumerate}
\item $F^\natural = I_X$.
\item For $1 \le i \le d$, $F_i = E_i$.
\item For $1 \le i \le d$, $F_i^* = E_i^*$.
\end{enumerate}
\end{lem}
\begin{proof}
Routine using Lemma \ref{lem2} (v), (vi), (vii).
\end{proof}

Let $\mathcal{F}, \mathcal{F}^*$ denote the linear subspace of the primary subalgebra $\mathcal{T}_0$ spanned by $\{F_i\}_{i=0}^d$, $\{F_i^*\}_{i=0}^d$, respectively.

\begin{lem}\label{lemF}
With above notation, we have the following.
\begin{enumerate}
\item $\mathcal{F} \neq \{O_X\}$ and $\mathcal{F}^* \neq \{O_X\}$.
\item $\mathcal{F}, \mathcal{F}^*$ are subalgebras of the primary subalgebra $\mathcal{T}_0$.
\item Each $\{F_i\}_{i=0}^d$, $\{F_i^*\}_{i=0}^d$ forms the primitive idempotent basis for $\mathcal{F}, \mathcal{F}^*$, respectively.
\item The primary subalgebra $\mathcal{T}_0$ is generated by $\mathcal{F}$ and $\mathcal{F}^*$.
\end{enumerate}
\end{lem}
\begin{proof}
(i) follows from Lemma \ref{lem2} (i), (ii).
(ii) and (iii) follow from Lemma \ref{lem2} (v), (vi).
(iv) follows from Lemma \ref{lem2} (iii), (iv).
\end{proof}

For $1 \le i \le d$, we define the following matrices in the Terwilliger algebra $\mathcal{T}$:
\begin{align*}
G_i = E_i - F_i,
&&
G_i^* = E_i^* - F_i^*.
\end{align*}
In addition, we define $G^\natural = I_X - F^\natural$.
The following lemma is easily obtained.

\begin{lem}\label{lem3}
With above notation, for $1 \le i,j \le d$ and $0 \le h \le d$, we have the following.
\begin{enumerate}
\item $G_iG_j = \delta_{i,j}G_i$.
\item $G_i^*G_j^* = \delta_{i,j}G_i^*$.
\item $G_iF_h = F_hG_i = O_X$.
\item $G_iF_h^* = F_h^*G_i = O_X$.
\item $G_i^*F_h = F_hG_i^* = O_X$.
\item $G_i^*F_h^* = F_h^*G_i^* = O_X$.
\item $G_iG_j^* = E_iE_j^* - F_iF_j^*$.
\item $G_i^*G_j = E_i^*E_j - F_i^*F_j$.
\item $G^\natural = G_1 + G_2 + \cdots + G_d = G_1^* + G_2^* + \cdots + G_d^*$.
\end{enumerate}
\end{lem}
\begin{proof}
(i) and (iii) follow from (E4) and Lemma \ref{lem2} (i), (v).
(ii) and (vi) follow from (E4*) and Lemma \ref{lem2} (ii), (vi).
(iv) follows from (iii) and Proposition \ref{T0}.
(v) follows from (vi) and Proposition \ref{T0}.
(vii) and (viii) follow from (iv) and (v).
(ix) follows from (E2), (E2*) and Lemma \ref{lem2} (vii).
\end{proof}

\begin{lem}\label{lem:eqG}
With above notation, the following are equivalent.
\begin{enumerate}
\item $G^\natural = O_X$.
\item For $1 \le i \le d$, $G_i = O_X$.
\item For $1 \le i \le d$, $G_i^* = O_X$.
\end{enumerate}
Moreover, these are equivalent to the equivalent conditions in Lemma \ref{lem:eqF}
\end{lem}
\begin{proof}
Routine using Lemma \ref{lem3} (i), (ii), (ix).
The last assertion follows from the definitions of $G_i, G_i^*, G^\natural$ and Lemma \ref{lem:eqF}.
\end{proof}

Let $\mathcal{G}, \mathcal{G}^*$ denote the linear subspace of the Terwilliger algebra $\mathcal{T}$ spanned by $\{G_j\}_{j=1}^d$, $\{G_j^*\}_{j=1}^d$, respectively.

\begin{lem}\label{lemG}
With above notation, suppose the equivalent conditions in Lemma \ref{lem:eqG} do not hold.
Then the following hold.
\begin{enumerate}
\item $\mathcal{G} \neq \{O_X\}$ and $\mathcal{G}^* \neq \{O_X\}$.
\item $\mathcal{G}, \mathcal{G}^*$ are subalgebras of the Terwilliger algebra $\mathcal{T}$.
\item Each $\{G_j\}_{j=1}^d$, $\{G_j^*\}_{j=1}^d$ forms the primitive idempotent basis for $\mathcal{G}, \mathcal{G}^*$, respectively.
\item The element $G^\natural$ is the common identity element in $\mathcal{G}$ and $\mathcal{G}^*$.
\item $\mathcal{F} \cap \mathcal{G} = \mathcal{F}^* \cap \mathcal{G} = \mathcal{F} \cap \mathcal{G}^* = \mathcal{F}^* \cap \mathcal{G}^* = \{O_X\}$.
\end{enumerate}
\end{lem}
\begin{proof}
For (i), by Lemma \ref{lem3} (ix), $G^\natural$ is the common element in $\mathcal{G}$ and $\mathcal{G}^*$ which is not $O_X$ since the condition (ii) in Lemma \ref{lem:eqG} does not hold.
(ii) and (iii) follow from Lemma \ref{lem3} (i), (ii).
(iv) follows from Lemma \ref{lem3} (i), (ii), (ix).
(v) follows from Lemma \ref{lem3} (i)--(vi).
\end{proof}

\begin{prop}\label{prop}
Let $\mathcal{T}$ be the Terwilliger algebra.
Let $\mathcal{T}_0$ denote the primary subalgebra of $\mathcal{T}$ generated by $\mathcal{F}, \mathcal{F}^*$.
Let $\mathcal{T}_1$ denote the subalgebra of $\mathcal{T}$ generated by $\mathcal{G}, \mathcal{G}^*$.
Then
\[
\mathcal{T} = \mathcal{T}_0 \oplus \mathcal{T}_1.
\]
In patricular, if the equivalent conditions in Lemmas \ref{lem:eqF} and \ref{lem:eqG} do not hold, then $\mathcal{T}$ coincides with the primary subalgebra $\mathcal{T}_0$.
\end{prop}
\begin{proof}
Let $\mathcal{T}'$ denote the subalgebra of $\mathcal{T}$ generated by $\mathcal{T}_0$ and $\mathcal{T}_1$.
By Lemma \ref{lem3} (iii)--(vi), $\mathcal{T}'$ is a direct sum of $\mathcal{T}_0$ and $\mathcal{T}_1$.
Lemma \ref{lem2} (i) and (ii), along with the definitions of $G_i$ and $G_i^*$,
show that 
$E_0 = F_0, E_0^* = F_0^*$, $E_i = F_i + G_i$ and $E_i^* = F_i^* + G_i^*$ for $1 \le i \le d$ are in $\mathcal{T}'$.
This means $\mathcal{T}$ is included in $\mathcal{T}'$.
Given that $\mathcal{T}'$ is a subalgebra of $\mathcal{T}$, they must be the same. 
For the last assertion,
observe that
if the equivalent conditions in Lemmas \ref{lem:eqF} and \ref{lem:eqG} do not hold,
$\mathcal{T}_1 = \{O_X\}$.
Thus, the result follows.
\end{proof}

\section{Ordered Hamming schemes}\label{OHS}

We define the ordered Hamming scheme in this section, which is the main subject of this paper. The notation introduced will be used throughout the rest of the paper.
Fix a positive integer $m, n$ and fix an $m$-sequence $(q_1, q_2, \ldots, q_m)$ of positive integers at least $2$.
We write $X_j = \{0,1,\ldots, q_j-1\}$ for $1 \le j \le m$ and set $X = X_1 \times X_2 \times \cdots \times X_m$.
We consider the $n$-ary Cartesian power $X^n$, whose element will be presented by
$x = (\vec{x}_1, \vec{x}_2, \ldots, \vec{x}_n)$,
where $\vec{x}_i = (x_{i,1}, x_{i,2}, \ldots, x_{i,m})$ with $x_{i,j} \in X_j$
for $1 \le i \le n$, $1 \le j \le m$.
The shape of $x \in X^n$ is an $(m+1)$-sequence $(\lambda_0, \lambda_1, \ldots, \lambda_m)$ such that
$\lambda_0 = |\{ i \mid 0 \le i \le n, \vec{x}_i = (0,0,\ldots,0)\}|$ and
\begin{align*}
\lambda_j = |\{  i \mid 0 \le i \le n, x_{i,j} \neq 0, x_{i,j+1} = x_{i,j+2} = \cdots = x_{i,m} = 0\}|
&&
1 \le j \le m.
\end{align*}
Let $\mathbb{I}(m,n)$ denote the set of all shapes of $x \in X^n$.
In other words,
\[
\mathbb{I}(m,n) = \{(\lambda_0, \lambda_1, \ldots, \lambda_m) \in \{0,1,\ldots, n\}^{m+1} \mid \lambda_0 + \lambda_1 + \cdots + \lambda_m = n\}.
\]
For $\lambda \in \mathbb{I}(m,n)$, let $R_\lambda$ denote the $\lambda$-th associated relation in $X^n \times X^n$, consisting of all $(x,y) \in X^n \times X^n$ where $x-y$ has shape $\lambda$.
Then the pair $\mathfrak{X}(m,n) = \mathfrak{X}(m,n;q_1, q_2, \ldots, q_m) = (X^n, \{R_\lambda\}_{\lambda \in \mathbb{I}(m,n)})$ forms a symmetric association scheme.
We refer $\mathfrak{X}(m,n)$ as the ordered Hamming scheme.
The number of relations are given by
\begin{equation}
|\mathbb{I}(m,n)| = \binom{m+n}{n}, \label{relations}
\end{equation}
which is the number of multisets of cardinality $n$ taken from a set of size $m+1$.
We remark that 
$\mathfrak{X}(1,n;q)$ is the (ordinary) Hamming scheme $H(m,q)$ and that
$\mathfrak{X}(m,1; q_1, q_2, \ldots, q_m)$ is the wreath product of the one-class association schemes $H(1,q_1) \wr H(1,q_2) \wr \cdots \wr H(1,q_m)$.

\section{Bose-Mesner algebra of $\mathfrak{X}(m,1)$}
We determine the structure of the Bose-Mesner algebra of $\mathfrak{X}(m,1) = \mathfrak{X}(m,1; q_1, q_2, \ldots, q_m)$, which is the wreath product of the one-class association schemes $H(1,q_1) \wr H(1,q_2) \wr \cdots \wr H(1,q_m)$.
To read indices more easily, we write $I_j = I_{X_j}$, $J_j = J_{X_j}$ for $1 \le j \le m$.
Put $\tilde{J}_j = \frac{1}{q_j}J_j$ for $1 \le j \le m$.
Since each sequence in $\mathbb{I}(m,1)$ has all $0$ except one that equals $1$,
we can identify the sequence in $\mathbb{I}(m,1)$ whose $j$-th element is $1$ with an integer $j$.
The $j$-th adjacency matrix $\boldsymbol{A}_j$ of $\mathfrak{X}(m,1)$ is given by
\begin{equation}
\boldsymbol{A}_j = 
\begin{cases}
I_1 \otimes I_2 \otimes \cdots \otimes I_m & \text{if $j = 0$}, \\
J_1 \otimes J_2 \otimes \cdots \otimes J_{j-1} \otimes (J_j - I_j) \otimes I_{j+1} \otimes I_{j+2} \otimes \cdots \otimes I_m & \text{if $1 \le j \le m$}.
\end{cases}
\label{A}
\end{equation}

\begin{prop}
For $0 \le j \le m$, 
the $j$-th primitive idempotent of $\mathfrak{X}(m,1)$ can be given by
\begin{equation}
\boldsymbol{E}_j = 
\begin{cases}
\tilde{J}_1 \otimes \tilde{J}_2 \otimes \cdots \otimes \tilde{J}_m & \text{if $j = 0$}, \\
\tilde{J}_1 \otimes \tilde{J}_2 \otimes \cdots \otimes \tilde{J}_{m-j} \otimes (I_{m-j+1} - \tilde{J}_{m-j+1}) \otimes I_{m-j+2} \otimes I_{m-j+3} \otimes \cdots \otimes I_m & \text{if $1 \le j \le m$}.
\end{cases}
\label{E}
\end{equation}
\end{prop}
\begin{proof}
Straightforward to check the four defining conditions (E1)--(E4) for the primitive idempotents.  
\end{proof}

\begin{prop}\label{km}
For $0 \le j \le m$, the $j$-th valency $k_j$ and 
the $j$-th multiplicity $m_j$ of $\mathfrak{X}(m,1)$ are given as follows.
\begin{align*}
k_j =
\begin{cases}
1 & \text{if $j = 0$}, \\
(q_j-1)\prod_{s=1}^{j-1} q_s & \text{if $1 \le j \le m$},
\end{cases}
&&
m_j =
\begin{cases}
1 & \text{if $j = 0$}, \\
(q_{m-j+1}-1)\prod_{s=m-j+2}^m q_s & \text{if $1 \le j \le m$}.
\end{cases}
\end{align*}
\end{prop}
\begin{proof}
Recall $k_j$ is the constant row sum of $\boldsymbol{A}_j$, which is calculated directly from \eqref{A}.
Similarly, $m_j$ is the trace of $\boldsymbol{E}_j$ and it is calculated directly from \eqref{E}.
\end{proof}

\begin{cor}\label{k'm'}
Fix $0 \le j \le m$.
Let $k_j$, $m_j$ denote the $j$-th valency and the $j$-th multiplicity of $\mathfrak{X}(m,1; q_1, q_2, \ldots, q_m)$, respectively.
Let $k_j'$, $m_j'$ denote the $j$-th valency and the $j$-th multiplicity of $\mathfrak{X}(m,1; q_m, q_{m-1}, \ldots, q_1)$, respectively.
Then
$k_j' = m_j$ and $m_j' = k_j$.
\end{cor}
\begin{proof}
Routine using Proposition \ref{km}.
\end{proof}
%$A_jE_i = P_j(i)E_i$
\begin{prop}\label{PQ}
For $0 \le j \le m$,
let $k_j$, $m_j$ denote the $j$-th valency and the $j$-th multiplicity of $\mathfrak{X}(m,1)$, respectively.
For $0 \le i,j \le m$, the $(i,j)$-entries of the first eigenmatrix $P$ and the second eigenmatrix $Q$ of $\mathfrak{X}(m,1)$ are give by
\begin{align*}
P_j(i) = 
\begin{cases}
k_j & \text{if $i + j < m + 1$},\\
- \dfrac{k_j}{q_j-1} & \text{if $i + j = m + 1$}, \\
0 & \text{if $i + j > m+1$},
\end{cases}
&&
Q_j(i) = 
\begin{cases}
m_j & \text{if $i + j < m + 1$},\\
- \dfrac{m_j}{q_{m-j+1}-1} & \text{if $i + j = m + 1$}, \\
0 & \text{if $i + j > m+1$}.
\end{cases}
\end{align*}
\end{prop}
\begin{proof}
Direct calculation using \eqref{A} and \eqref{E}.
\end{proof}

\begin{cor}\label{P'Q'}
Let $P$, $Q$ denote the first eigenmatrix and the second eigenmatrix of $\mathfrak{X}(m,1; q_1, q_2, \ldots, q_m)$, respectively.
Let $P'$, $Q'$ denote the first eigenmatrix and the second eigenmatrix of $\mathfrak{X}(m,1; q_m, q_{m-1}, \ldots, q_1)$, respectively.
Then
$P' = Q$ and $Q' = P$.

In particular, if $(q_1,q_2,\ldots,q_m) = (q_m, q_{m-1}, \ldots, q_1)$, 
the association scheme $\mathfrak{X}(m,1; q_1, q_2, \ldots, q_m)$ is self-dual.
\end{cor}
\begin{proof}
Routine using Corollary \ref{k'm'} and Proposition \ref{PQ}.
\end{proof}

\section{Bose-Mesner algebra of $\mathfrak{X}(m,n)$}
We determine the structure of the Bose-Mesner algebra of $\mathfrak{X}(m,n) = \mathfrak{X}(m,n; q_1, q_2, \ldots, q_m)$.
Recall the adjacency matrices $\{\boldsymbol{A}_j\}_{j=0}^m$ and the primitive idempotents $\{\boldsymbol{E}_j\}_{j=0}^m$ of $\mathfrak{X}(m,1)$ from \eqref{A} and \eqref{E}.
For $\lambda \in \mathbb{I}(m,n)$, the $\lambda$-th adjacency matrix of $\mathfrak{X}(m,n)$ is given by
\begin{equation}\label{vA}
\mathscr{L}_\lambda(\boldsymbol{A}_\bullet) = \mathscr{L}_\lambda(\boldsymbol{A}_0, \boldsymbol{A}_1, \ldots, \boldsymbol{A}_m).
\end{equation}
Note that if $\lambda = (n,0,0,\ldots,0) \in \mathbb{I}(m,n)$, then $\mathscr{L}_\lambda(\boldsymbol{A}_\bullet) = (\boldsymbol{A}_0)^{\otimes n}$ is the identity.
\begin{prop}\label{BM}
Let $\mathcal{A}$ denote the Bose-Mesner algebra of $\mathfrak{X}(m,1)$.
Then the Bose-Mesner algebra of $\mathfrak{X}(m,n)$ is the $n$-fold symmetric tensor algebra $\operatorname{Sym}^n(\mathcal{A})$.
\end{prop}
\begin{proof}
Routine using the definition of the Bose-Mesner algebra, Lemma \ref{sym} and \eqref{vA}.
\end{proof}

\begin{prop}
For $\lambda \in \mathbb{I}(m,n)$
the $\lambda$-th primitive idempotent of $\mathfrak{X}(m,n)$ can be given by
\begin{equation}\label{vE}
\mathscr{L}_\lambda(\boldsymbol{E}_\bullet) = \mathscr{L}_\lambda(\boldsymbol{E}_0, \boldsymbol{E}_1, \ldots, \boldsymbol{E}_m).
\end{equation}
\end{prop}
\begin{proof}
Straightforward to check the four defining conditions (E1)--(E4) for the primitive idempotents.  
\end{proof}

\begin{prop}\label{vkm}
For $\lambda = (\lambda_0, \lambda_1, \ldots, \lambda_m) \in \mathbb{I}(m,n)$, the $\lambda$-th valency and 
the $\lambda$-th multiplicity of $\mathfrak{X}(m,n)$ are given as follows.
\begin{align*}
\binom{n}{\lambda_0, \lambda_1, \ldots, \lambda_m} \prod_{j=0}^m k_j^{\lambda_j},
&&
\binom{n}{\lambda_0, \lambda_1, \ldots, \lambda_m} \prod_{j=0}^m m_j^{\lambda_j},
\end{align*}
where $k_j, m_j$ are the $j$-th valency and the $j$-th multiplicity of $\mathfrak{X}(m,1)$, respectively.
\end{prop}
\begin{proof}
Direct calculation using \eqref{vA} and \eqref{vE}.
\end{proof}

\begin{cor}\label{vk'm'}
Fix $\lambda \in \mathbb{I}(m,n)$.
Let $k_\lambda$, $m_\lambda$ denote the $\lambda$-th valency and the $\lambda$-th multiplicity of $\mathfrak{X}(m,n; q_1, q_2, \ldots, q_m)$, respectively.
Let $k_\lambda'$, $m_\lambda'$ denote the $\lambda$-th valency and the $\lambda$-th multiplicity of $\mathfrak{X}(m,n; q_m, q_{m-1}, \ldots, q_1)$, respectively.
Then
$k_\lambda' = m_\lambda$ and $m_\lambda' = k_\lambda$.
\end{cor}
\begin{proof}
Routine using Corollary \ref{k'm'} and Proposition \ref{vkm}.
\end{proof}

\begin{lem}\label{gf}
As matrices of polynomials in the indeterminate $z_0, z_1, \ldots, z_m$, we have
\begin{align*}
\left( \sum_{i=0}^m \boldsymbol{A}_iz_i\right)^{\otimes n} = \sum_{\lambda \in \mathbb{I}(m,n)} \mathscr{L}_\lambda(\boldsymbol{A}_\bullet) z^\lambda,
&&
\left( \sum_{i=0}^m \boldsymbol{E}_iz_i\right)^{\otimes n} = \sum_{\lambda \in \mathbb{I}(m,n)} \mathscr{L}_\lambda(\boldsymbol{E}_\bullet) z^\lambda,
\end{align*}
where $z^\lambda = \prod_{i=0}^m z_i^{\lambda_i}$ if $\lambda = (\lambda_0, \lambda_1, \ldots, \lambda_m)$.
\end{lem}
\begin{proof}
We only provide the proof for the first equality since the other is derived similarly.
In the left-hand side of the formula,
the coefficient of $z^\lambda$ is contained exactly $\lambda_i$ numbers of $\boldsymbol{A}_i$ for $0 \le i \le m$ and
each combination appears exactly once.
This means the coefficient is $\mathscr{L}_\lambda(\boldsymbol{A}_\bullet)$.
The result follows.
\end{proof}

\begin{defi}\label{Krawchouk}
For $0 \le i,j \le m$, let $P_j(i)$ be the $(i,j)$ entry of the first eigenmatrix of $\mathfrak{X}(m,1;q_1, q_2, \ldots, q_m)$.
Let $\lambda = (\lambda_0, \lambda_1, \ldots, \lambda_m) \in \mathbb{I}(m,n)$.
The multivariate Krawchouk polynomials $K_\mu(\lambda;q_1, q_2, \ldots, q_m)$ can be defined by means of the generating function
\[
\prod_{j=0}^m\left( \sum_{i=0}^m P_i(j) z_i\right)^{\lambda_j} = \sum_{\mu \in \mathbb{I}(m,n)} K_\mu(\lambda;q_1, q_2, \ldots, q_m) z^\mu,
\]
where $z^\mu = \prod_{i=0}^m z_i^{\mu_i}$ if $\mu = (\mu_0, \mu_1, \ldots, \mu_m)$.
\end{defi}

\begin{lem}\label{Krawchouk2}
For $0 \le i,j \le m$, let $Q_j(i)$ be the $(i,j)$ entry of the second eigenmatrix of $\mathfrak{X}(m,1;q_1, q_2, \ldots, q_m)$.
For $\lambda = (\lambda_0, \lambda_1, \ldots, \lambda_m) \in \mathbb{I}(m,n)$, as polynomials in the indeterminate $z_0, z_1, \ldots, z_m$, we have
\[
\prod_{j=0}^m\left( \sum_{i=0}^m Q_i(j) z_i\right)^{\lambda_j} = \sum_{\mu \in \mathbb{I}(m,n)} K_\mu(\lambda;q_m, q_{m-1}, \ldots, q_1) z^\mu,
\]
where $z^\mu = \prod_{i=0}^m z_i^{\mu_i}$ if $\mu = (\mu_0, \mu_1, \ldots, \mu_m)$.
\end{lem}
\begin{proof}
Routine using Corollary \ref{P'Q'} and Definition \ref{Krawchouk}.
\end{proof}

%$A_\muE_\lambda = P_\mu(\lambda)(i)E_\lambda$
\begin{prop}\label{vPQ}
For $\lambda, \mu \in \mathbb{I}(m,n)$, the $(\lambda,\mu)$-entries of the first eigenmatrix $P$ and the second eigenmatrix $Q$ of $\mathfrak{X}(m,n)$ are expressed in terms of the multivariate Krawthouk polynomials:
\begin{align*}
K_\mu(\lambda;q_1, q_2, \ldots, q_m),
&&
K_\mu(\lambda;q_m, q_{m-1}, \ldots, q_1).
\end{align*}
\end{prop}
\begin{proof}
For $0 \le i,j \le m$, let $P_j(i)$ be the $(i,j)$-entry of the first eigenmatrix of $\mathfrak{X}(m,1)$.
By Lemma \ref{gf} and Definition \ref{Krawchouk}, as matrices of polynomials in the indeterminate $z_0, z_1, \ldots, z_m, z_0', z'_1, \ldots, z'_m$,
we have
\begin{align*}
\left(\sum_{\mu \in \mathbb{I}(m,n)} \mathscr{L}_\mu(\boldsymbol{A}_\bullet) z^\mu\right)\left(\sum_{\lambda \in \mathbb{I}(m,n)} \mathscr{L}_\lambda(\boldsymbol{E}_\bullet) z'^\lambda\right)
&= \left(\sum_{i=0}^m\sum_{j=0}^m \boldsymbol{A}_i\boldsymbol{E}_jz_iz_j'\right)^{\otimes n} \\
&= \left\{\sum_{j=0}^m \left(\sum_{i=0}^mP_i(j)z_i \right)\boldsymbol{E}_jz_j'\right\}^{\otimes n} \\
&= \sum_{\lambda = (\lambda_0, \lambda_1, \ldots, \lambda_m) \in \mathbb{I}(m,n)}\left\{\prod_{j=0}^m \left(\sum_{i=0}^m P_i(j)z_i\right)^{\lambda_j} \right\}\mathscr{L}_\lambda(\boldsymbol{E}_\bullet) z'^\lambda \\
&= \sum_{\lambda \in \mathbb{I}(m,n)} \left(\sum_{\mu \in \mathbb{I}(m,n)} K_\mu(\lambda;q_1, q_2, \ldots, q_m) z^\mu \right) \mathscr{L}_\lambda(\boldsymbol{E}_\bullet) z'^\lambda,
\end{align*}
where $z^\mu = \prod_{i=0}^m z_i^{\mu_i}$ if $\mu = (\mu_0, \mu_1, \ldots, \mu_m)$
and $z'^\lambda = \prod_{i=0}^m z_i'^{\lambda_i}$ if $\lambda = (\lambda_0, \lambda_1, \ldots, \lambda_m)$.
By comparing the coefficient of $z^\mu z'^\lambda$, we obtain
$\mathscr{L}_\mu(\boldsymbol{A}_\bullet) \mathscr{L}_\lambda(\boldsymbol{E}_\bullet) = 
K_\lambda(\mu;q_1, q_2, \ldots, q_m) \mathscr{L}_\lambda(\boldsymbol{E}_\bullet)$ as desired.

Similarly, by Lemmas \ref{gf} and \ref{Krawchouk2}, we obtain
\[
\left(\sum_{\mu \in \mathbb{I}(m,n)} \mathscr{L}_\mu(\boldsymbol{E}_\bullet) z^\mu\right) \circ \left(\sum_{\lambda \in \mathbb{I}(m,n)} \mathscr{L}_\lambda(\boldsymbol{A}_\bullet) z'^\lambda\right) =\sum_{\lambda \in \mathbb{I}(m,n)}  \left(\sum_{\mu \in \mathbb{I}(m,n)} K_\mu(\lambda;q_m, q_{m-1}, \ldots, q_1) z^\mu \right)\mathscr{L}_\lambda(\boldsymbol{A}_\bullet) z'^\lambda,
\]
By comparing the coefficient of $z^\mu z'^\lambda$, we obtain
$\mathscr{L}_\mu(\boldsymbol{E}_\bullet) \circ \mathscr{L}_\lambda(\boldsymbol{A}_\bullet) = K_\mu(\lambda;q_m, q_{m-1}, \ldots, q_1) \mathscr{L}_\lambda(\boldsymbol{A}_\bullet)$ as desired.
\end{proof}

\begin{cor}\label{vP'Q'}
Let $P$, $Q$ denote the first eigenmatrix and the second eigenmatrix of $\mathfrak{X}(m,n; q_1, q_2, \ldots, q_m)$, respectively.
Let $P'$, $Q'$ denote the first eigenmatrix and the second eigenmatrix of $\mathfrak{X}(m,n; q_m, q_{m-1}, \ldots, q_1)$, respectively.
Then
$P' = Q$ and $Q' = P$.

In particular, if $(q_1,q_2,\ldots,q_m) = (q_m, q_{m-1}, \ldots, q_1)$, 
the association scheme $\mathfrak{X}(m,n; q_1, q_2, \ldots, q_m)$ is self-dual.
\end{cor}

\begin{proof}
Routine using Proposition \ref{vPQ}.
\end{proof}

\section{Terwilliger algebra of $\mathfrak{X}(m,1)$}

We determine the structure of the Terwilliger algebra $\mathcal{T}(m,1)$ of the wreath product $\mathfrak{X}(m,1; q_1, q_2, \ldots, q_m)$ with respect to $0 = (0,0,\ldots,0) \in X^n$.
The results in this section have already been studied by Bhattacharyya-Song-Tanaka \cite{BST}.
For the sake of completeness and to introduce some notations used later in this paper, we reprove the known results from a matrix algebraic perspective.
As previous section, let $I_j$, $J_j$ denote the identity matrix, the all-ones matrix in $\operatorname{Mat}_{X_j}(\mathbb{C})$
and $\tilde{J}_j = \frac{1}{q_j}J_j$.
Moreover, let $D_j = D_j(0)$ denote the $(0,0)$ matrix unit in $\operatorname{Mat}_{X_j}(\mathbb{C})$.
Recall the adjacency matrices $\{\boldsymbol{A}_j\}_{j=0}^m$ and the primitive idempotents $\{\boldsymbol{E}_j\}_{j=0}^m$ of $\mathfrak{X}(m,1)$ from \eqref{A} and \eqref{E}.
We also recall the valencies $\{k_j\}_{j=0}^m$ and the multiplicities $\{m_j\}_{j=0}^m$ from Proposition \ref{km}.

\begin{prop}
For $0 \le j \le m$, 
the $j$-th dual primitive idempotent $\boldsymbol{E}^*_j = \boldsymbol{E}^*_j(0)$ with respect to $0 \in X^n$
is given by
\begin{equation}
\boldsymbol{E}^*_j = 
\begin{cases}
D_1 \otimes D_2 \otimes \cdots \otimes D_m & \text{if $j = 0$}, \\
I_1 \otimes I_2 \otimes \cdots \otimes I_{j-1} \otimes (I_j - D_j) \otimes D_{j+1} \otimes D_{j+2} \otimes \cdots \otimes D_m & \text{if $1 \le j \le m$}.
\end{cases}
\label{E*}
\end{equation}
\end{prop}
\begin{proof}
Routine using \eqref{A}.
\end{proof}

For $0 \le j \le m$, we define
\begin{align}\label{FF*}
\boldsymbol{F}_j = \frac{|X|}{m_j}\boldsymbol{E}_j\boldsymbol{E}^*_0\boldsymbol{E}_j,
&&
\boldsymbol{F}_j^* = \frac{|X|}{k_j}\boldsymbol{E}^*_j\boldsymbol{E}_0\boldsymbol{E}^*_j.
\end{align}
Let $\mathcal{F}, \mathcal{F}^*$ denote the vector space spanned by $\{\boldsymbol{F}_j\}_{j=0}^m$, $\{\boldsymbol{F}_j^*\}_{j=0}^m$, respectively.
By Lemma \ref{lemF},
$\mathcal{F}, \mathcal{F}^*$ are subalgebras of the primary subalgebra of the Terwilliger algebra of $\mathfrak{X}(m,1)$
and 
each $\{\boldsymbol{F}_j\}_{j=0}^m$, $\{\boldsymbol{F}_j^*\}_{j=0}^m$ forms the primitive idempotent basis for $\mathcal{F}, \mathcal{F}^*$, respectively.

If we define matrices in $\operatorname{Mat}_{X_j}(\mathbb{C})$ by
\begin{align*}
H_j = \frac{q_j}{q_j-1}(I_j-\tilde{J}_j)D_j(I_j-\tilde{J}_j),
&&
H_j^* = \frac{q_j}{q_j-1}(I_j-D_j)\tilde{J}_j(I_j-D_j)
\end{align*}
for $1 \le j \le m$,
then from \eqref{E} and \eqref{E*}, we get
\begin{align}
\boldsymbol{F}_j &= \tilde{J}_1 \otimes \tilde{J}_2 \otimes \cdots \otimes \tilde{J}_{j-1} \otimes H_j \otimes D_{j+1} \otimes D_{j+2} \otimes \cdots \otimes D_m,\label{F}\\
\boldsymbol{F}_j^* &= \tilde{J}_1 \otimes \tilde{J}_2 \otimes \cdots \otimes \tilde{J}_{m-j} \otimes H_{m-j+1}^* \otimes D_{m-j+2} \otimes D_{m-j+3} \otimes \cdots \otimes D_m.\label{F*}
\end{align}

\begin{lem}\label{m=1,q=2}
The following are equivalent.
\begin{enumerate}
\item $m=1$ and $q_1 = 2$.
\item For $1 \le j \le m$, $\boldsymbol{F}_j = \boldsymbol{E}_j$.
\item For $1 \le j \le m$, $\boldsymbol{F}_j^* = \boldsymbol{E}_j^*$.
\end{enumerate}
\end{lem}
\begin{proof}
Routine using the expressions \eqref{F} and \eqref{F*}.
\end{proof}

\begin{prop}[c.f.\ {\cite[Theorem 4.2]{BST}}]\label{T0(m,1)}
The Terwilliger algebra of $\mathfrak{X}(m,1)$ coincides with the primary subalgebra if and only if $m=1$ and $q_1=2$.
\end{prop}
\begin{proof}
Routine using Proposition \ref{prop} and Lemma \ref{m=1,q=2}.
\end{proof}

Our next goal is to determine the Terwilliger algebra when $m \ge 2$ or $q_1 \ge 3$.

\begin{lem}\label{Xj}
Fix $1 \le j \le m$. In the $\operatorname{Mat}_{X_j}(\mathbb{C})$, we have the following identities.
Here we omit the subscript $j$.
\begin{enumerate}
\item $\tilde{J}H = H\tilde{J} = 0$.
\item $DH^* = H^*D = 0$.
\item $D(I-\tilde{J}) = DH$ and $(I-\tilde{J})D = HD$.
\item $\tilde{J}(I-D) = \tilde{J}H^*$ and $(I-D)\tilde{J} = H^*\tilde{J}$.
\item $(I-\tilde{J})(I-D) - HH^* = (I-D)(I-\tilde{J}) - H^*H = I-\tilde{J} - H = I - D - H^* = 0$ if $q = 2$.
\item $(I-\tilde{J})(I-D) - HH^* = (I-D)(I-\tilde{J}) - H^*H = I-\tilde{J} - H = I - D - H^* \neq 0$ if $q \ge 3$.
\end{enumerate}
\end{lem}
\begin{proof}
(i)--(vi) except the assertion the matrix in (vi) is non-zero follow from direct calculation.
The matrix in (vi) is non-zero because its trace is $q-2$, which is non-zero if $q \ge 3$.
\end{proof}

\begin{lem}\label{EF}
For $1 \le i,j \le m$, we have the following.
\begin{enumerate}
\item If $i+j < m+1$ or $i+j = m+1$ with $q_i = 2$,
$\boldsymbol{E}_j\boldsymbol{E}_i^* = \boldsymbol{F}_j\boldsymbol{F}_i^*$ and $\boldsymbol{E}_i^*\boldsymbol{E}_j = \boldsymbol{F}_i^*\boldsymbol{F}_j$.
\item If $i+j > m+1$, 
$\boldsymbol{F}_j\boldsymbol{F}_i^* = 0$ and $\boldsymbol{F}_i^*\boldsymbol{F}_j = 0$.
\end{enumerate}
\end{lem}
\begin{proof}
\begin{enumerate}
\item The case $i+j<m+1$ follows from Lemma \ref{Xj} (iii), (iv) and the expressions \eqref{E}, \eqref{E*}, \eqref{F} and \eqref{F*}.
The case $i+j = m+1$ with $q_i = 2$ follows from Lemma \ref{Xj} (v) and the expressions \eqref{E}, \eqref{E*}, \eqref{F} and \eqref{F*}.
\item The $i$-th factor of $\boldsymbol{F}_j\boldsymbol{F}_i^*$ is $D_iH^*_i = 0$ from Lemma \ref{Xj} (ii), which implies 
$\boldsymbol{F}_j\boldsymbol{F}_i^* = 0$.
Similarly, the $i$-th factor of $\boldsymbol{F}_j\boldsymbol{F}_i^*$ is $H^*_iD_i = 0$ from Lemma \ref{Xj} (ii), which implies 
$\boldsymbol{F}_j\boldsymbol{F}_i^* = 0$.
\end{enumerate}
\end{proof}

For the rest of this section, unless otherwise stated we assume $m \ge 2$ or $q_1 \ge 3$.
For $1 \le j \le m$, we define
\begin{align}\label{GG*}
\boldsymbol{G}_j = \boldsymbol{E}_j - \boldsymbol{F}_j,
&&
\boldsymbol{G}_j^* = \boldsymbol{E}_j^* - \boldsymbol{F}_j^*.
\end{align}
By Lemma \ref{m=1,q=2}, every matrices $\boldsymbol{G}_j$, $\boldsymbol{G}_j^*$ are not the zero matrix.
Let $\mathcal{G}, \mathcal{G}^*$ denote the vector space spanned by $\{\boldsymbol{G}_j\}_{j=1}^m$, $\{\boldsymbol{G}_j^*\}_{j=1}^m$, respectively.
By Lemma \ref{lemG},
$\mathcal{G}, \mathcal{G}^*$ are subalgebras of the Terwilliger algebra of $\mathfrak{X}(m,1)$
and 
each $\{\boldsymbol{G}_j\}_{j=1}^m$, $\{\boldsymbol{G}_j^*\}_{j=1}^m$ forms a basis for $\mathcal{G}, \mathcal{G}^*$, respectively.

\begin{lem}\label{commute}
Assume $m \ge 2$ or $q_1 \ge 3$.
For $1 \le i,j \le m$, we have the following.
\begin{enumerate}
\item If $i+j < m+1$ or $i+j = m+1$ with $q_i = 2$,
$\boldsymbol{G}_j\boldsymbol{G}_i^* = 0$ and $\boldsymbol{G}_i^*\boldsymbol{G}_j = 0$.
\item If $i+j = m+1$ with $q_i \ge 3$, 
\[
\boldsymbol{G}_j\boldsymbol{G}_i^* = \boldsymbol{G}_i^*\boldsymbol{G}_j = \tilde{J}_1 \otimes \tilde{J}_2 \otimes \cdots \otimes \tilde{J}_{i-1} \otimes Z_i \otimes D_{i+1} \otimes D_{i+2} \otimes \cdots \otimes D_m,
\]
where $Z_i = I_i - \tilde{J}_i - H_i = I_i - D_i - H_i^*$.
\item If $i+j > m+1$, 
\begin{align*}
\boldsymbol{G}_j\boldsymbol{G}_i^* = \boldsymbol{G}_i^*\boldsymbol{G}_j = \tilde{J}_1 \otimes \tilde{J}_2 \otimes \cdots \otimes \tilde{J}_{m-j} \otimes &(I_{m-j+1} - \tilde{J}_{m-j+1}) \otimes I_{m-j+2} \otimes I_{m-j+3} \otimes \cdots \\
&\cdots \otimes I_{i-1} \otimes (I_i - D_i) \otimes D_{i+1} \otimes D_{i+2} \otimes \cdots \otimes D_m.
\end{align*}
\end{enumerate}
\end{lem}
\begin{proof}
(i) follows from Lemma \ref{lem3} (vii), (viii) and Lemma \ref{EF} (i).
(ii) follows from Lemma \ref{lem3} (vii), (viii), Lemma \ref{Xj} (vi) and the expressions \eqref{E}, \eqref{E*}, \eqref{F} and \eqref{F*}.
(iii) follows from Lemma \ref{lem3} (vii), (viii), Lemma \ref{EF} (ii) and the expressions \eqref{E}, \eqref{E*}.
\end{proof}

\begin{nota}\label{Lambda}
Assume $m \ge 2$ or $q_1 \ge 3$.
Let $\Lambda(m)$ denote the set of pairs $(i,j)$ with $1 \le i,j \le m$ such that one of the following holds.
\begin{itemize}
\item $i+j > m+1$.
\item $i+j = m+1$ with $q_i \ge 3$.
\end{itemize}
\end{nota}

Note that, by Lemma \ref{commute}, $(i,j) \in \Lambda(m)$ if and only if $\boldsymbol{G}_j\boldsymbol{G}_i^* = \boldsymbol{G}_i^*\boldsymbol{G}_j \neq 0$.

\begin{lem}\label{Lambda(m)}
Referring to Notation \ref{Lambda},
$\Lambda(m) = \emptyset$.
\end{lem}
\begin{proof}
Immediate from definition.
\end{proof}

\begin{prop}[c.f.\ {\cite[Theorem 4.2]{BST}}]\label{T1(m,1)}
Assume $m \ge 2$ or $q_1 \ge 3$.
Referring to Notation \ref{Lambda},
let $\mathcal{T}_1$ denote the subalgebra of the Terwilliger algebra of $\mathfrak{X}(m,1)$ generated by $\mathcal{G}$ and $\mathcal{G}^*$.
Then $\mathcal{T}_1$ is a commutative algebra whose basis are
$\{\boldsymbol{G}_j\boldsymbol{G}_i^* \mid (i,j) \in \Lambda(m)\}$.
\end{prop}

\begin{proof}
Set $S = \{\boldsymbol{G}_j\boldsymbol{G}_i^* \mid (i,j) \in \Lambda(m)\}$.
By Lemma \ref{Lambda(m)}, we have $S \neq \emptyset$.
Let $\mathcal{T}_1'$ denote the linear subspace of $\mathcal{T}_1$ spanned by the set of matrices $S$.
By Lemma \ref{lem3} (i), (ii) and Lemma \ref{commute},
$S$ forms the primitive idempotent basis for $\mathcal{T}_1'$.
This means 
$\mathcal{T}_1'$ is a commutative subalgebra of $\mathcal{T}_1$.
By definition and by Lemma \ref{lem3} (i), (ix), for each $1 \le j \le m$, 
we have
\[
\boldsymbol{G}_j = \sum_{i=1}^m \boldsymbol{G}_j\boldsymbol{G}_i = \sum_{i=1}^m \boldsymbol{G}_j\boldsymbol{G}_i^* \in \mathcal{T}'.
\]
We obtain similarly that $\boldsymbol{G}_j^* \in \mathcal{T}'$.
Since $\mathcal{T}_1$ is generated by $\{\boldsymbol{G}_j\}_{j=1}^m$ and $\{\boldsymbol{G}_j^*\}_{j=1}^m$, $\mathcal{T}_1$ is contained in $\mathcal{T}_1'$
and so $\mathcal{T}_1 = \mathcal{T}_1'$.
The result follows.
\end{proof}

\begin{cor}[{\cite[Corollary 4.3]{BST}}]
The Terwilliger algebra of $\mathfrak{X}(m,1)$ is isomorphic, as $\mathbb{C}$-algebra, to
\[
\operatorname{Mat}_{m+1}(\mathbb{C}) \oplus \mathbb{C}^{m(m-1)/2 + \varepsilon},
\]
where $\varepsilon = |\{i \mid 1 \le i \le m, q_i \ge 3\}|$.
\end{cor}

\begin{proof}
Let $\mathcal{T}$ denote the Terwilliger algebra of $\mathfrak{X}(m,1)$.
If $m=1$ and $q_1=2$,  
by Propositions \ref{T0}, \ref{prop} and Lemma \ref{m=1,q=2},
$\mathcal{T}$ is isomorphic to $\operatorname{Mat}_{m+1}(\mathbb{C})$ as desired.
It remains to consider the case $m\ge2$ or $q_1 \ge 3$.

Let $\mathcal{T}_0$ denote the primary subalgebra of $\mathcal{T}$ generated by $\mathcal{F}$ and $\mathcal{F}^*$.
Let $\mathcal{T}_1$ denote the subalgebra of $\mathcal{T}$ generated by $\mathcal{G}$ and $\mathcal{G}^*$.
By Proposition \ref{prop}, $\mathcal{T}$ is a direct sum of $\mathcal{T}_0$ and $\mathcal{T}_1$.
By Proposition \ref{T0},
$\mathcal{T}_0$ is isomorphic, as $\mathbb{C}$-algebra, to $\operatorname{Mat}_{m+1}(\mathbb{C})$.
By Proposition \ref{T1(m,1)}, $\mathcal{T}_1$ is isomorphic to $\mathbb{C}^{|\Lambda(m)|}$.
Since $|\Lambda(m)| = m(m-1)/2 + \varepsilon$, the result follows.
\end{proof}

\section{Terwilliger algebra of $\mathfrak{X}(m,n)$}
We determine the structure of the Terwilliger algebra of $\mathfrak{X}(m,n) = \mathfrak{X}(m,n; q_1, q_2, \ldots, q_m)$.
Recall the primitive idempotents $\{\boldsymbol{E}_j\}_{j=0}^m$ and the dual primitive idempotents $\{\boldsymbol{E}_j^*\}_{j=0}^m$ of $\mathfrak{X}(m,1)$ from \eqref{E} and \eqref{E*}.
In this section, let $\boldsymbol{I}$ denote the identity, the $0$-th adjacency matrix, in the Bose-Mesner algebra of $\mathfrak{X}(m,1)$
and let $\boldsymbol{O}$ denote the zero matrix, in the Bose-Mesner algebra of $\mathfrak{X}(m,1)$.
Recall the primitive idempotents $\{\mathscr{L}_\lambda(\boldsymbol{E}_\bullet)\}_{\lambda \in \mathbb{I}(m,n)}$ of $\mathfrak{X}(m,n)$ from \eqref{vE}.

\begin{prop}
For $\lambda\in \mathbb{I}(m,n)$, the $\lambda$-th dual primitive idempotent of $\mathfrak{X}(m,n)$ can be given by
\begin{equation}\label{vE*}
\mathscr{L}_\lambda(\boldsymbol{E}_\bullet^*) = \mathscr{L}_\lambda(\boldsymbol{E}_0^*, \boldsymbol{E}_1^*, \ldots, \boldsymbol{E}_m^*).
\end{equation}
\end{prop}
\begin{proof}
Routine using \eqref{vA}.
\end{proof}

\begin{prop}\label{DBM}
Let $\mathcal{A}^*$ denote the dual Bose-Mesner algebra of $\mathfrak{X}(m,1)$.
Then the dual Bose-Mesner algebra of $\mathfrak{X}(m,n)$ is the $n$-fold symmetric tensor algebra $\operatorname{Sym}^n(\mathcal{A}^*)$.
\end{prop}
\begin{proof}
Routine using the definition of the dual Bose-Mesner algebra, Lemma \ref{sym}, and \eqref{vE*}.
\end{proof}

\begin{lem}\label{EI}
For $0 \le i \le m$,
$\mathscr{L}_{1,n-1}(\boldsymbol{E}_i, \boldsymbol{I})$ are in the Bose-Mesner algebra of $\mathfrak{X}(m,n)$
and
$\mathscr{L}_{1,n-1}(\boldsymbol{E}_i^*, \boldsymbol{I})$ are in the dual Bose-Mesner algebra of $\mathfrak{X}(m,n)$.
\end{lem}
\begin{proof}
Routine using Propositions \ref{BM} and \ref{DBM}.
\end{proof}

\begin{lem}\label{EEI}
For $0 \le i \le m$, the following symmetric tensors are in the Terwilliger algebra of $\mathfrak{X}(m,n)$.
\begin{align*}
\mathscr{L}_{1,n-1}(\boldsymbol{E}_0^*\boldsymbol{E}_i, \boldsymbol{I}),
&&
\mathscr{L}_{1,n-1}(\boldsymbol{E}_i\boldsymbol{E}_0^*, \boldsymbol{I}),
&&
\mathscr{L}_{1,n-1}(\boldsymbol{E}_0\boldsymbol{E}_i^*, \boldsymbol{I}),
&&
\mathscr{L}_{1,n-1}(\boldsymbol{E}_i^*\boldsymbol{E}_0, \boldsymbol{I}),
\end{align*}
\begin{align*}
\mathscr{L}_{1,n-1}(\boldsymbol{E}_i\boldsymbol{E}_0^*\boldsymbol{E}_i, \boldsymbol{I}),
&&
\mathscr{L}_{1,n-1}(\boldsymbol{E}_i^*\boldsymbol{E}_0\boldsymbol{E}_i^*, \boldsymbol{I}).
\end{align*}
\end{lem}
\begin{proof}
Let $\boldsymbol{\mathcal{T}}$ denote the Terwilliger algebra of $\mathfrak{X}(m,n)$.
Observe that
\begin{align*}
\mathscr{L}_{1,n-1}(\boldsymbol{E}_0^*, \boldsymbol{I})\mathscr{L}_{1,n-1}(\boldsymbol{E}_i, \boldsymbol{I}) &= \mathscr{L}_{1,n-1}(\boldsymbol{E}_0^*\boldsymbol{E}_i, \boldsymbol{I}) + \mathscr{L}_{1,1,n-2}(\boldsymbol{E}_0^*,\boldsymbol{E}_i, \boldsymbol{I}), \\
\mathscr{L}_{1,n-1}(\boldsymbol{E}_i, \boldsymbol{I})\mathscr{L}_{1,n-1}(\boldsymbol{E}_0^*, \boldsymbol{I}) &= \mathscr{L}_{1,n-1}(\boldsymbol{E}_i\boldsymbol{E}_0^*, \boldsymbol{I}) + \mathscr{L}_{1,1,n-2}(\boldsymbol{E}_0^*,\boldsymbol{E}_i, \boldsymbol{I}).
\end{align*}
By subtraction, since $\mathscr{L}_{1,n-1}(\boldsymbol{E}_0^*, \boldsymbol{I})$ and $\mathscr{L}_{1,n-1}(\boldsymbol{E}_i, \boldsymbol{I})$ are in $\boldsymbol{\mathcal{T}}$ from Lemma \ref{EI}, we obtain $\mathscr{L}_{1,n-1}(\boldsymbol{E}_0^*\boldsymbol{E}_i - \boldsymbol{E}_i\boldsymbol{E}_0^*, \boldsymbol{I}) \in \boldsymbol{\mathcal{T}}$.
Since there exists a non-zero scalar $c_i = |X|^{-1} m_i$, where $m_i$ is the $i$-th multiplicity of $\mathfrak{X}(m,1)$, such that $\boldsymbol{E}_0^*\boldsymbol{E}_i\boldsymbol{E}_0^* = c_i \boldsymbol{E}_0^*$ from Lemma \ref{lem}, we have
\begin{align*}
\mathscr{L}_{1,n-1}(\boldsymbol{E}_0^*, \boldsymbol{I})\mathscr{L}_{1,n-1}(\boldsymbol{E}_0^*\boldsymbol{E}_i - \boldsymbol{E}_i\boldsymbol{E}_0^*, \boldsymbol{I}) &= \mathscr{L}_{1,n-1}(\boldsymbol{E}_0^*\boldsymbol{E}_i - c_i\boldsymbol{E}_0^*, \boldsymbol{I}) + \mathscr{L}_{1,1,n-2}(\boldsymbol{E}_0^*,\boldsymbol{E}_0^*\boldsymbol{E}_i - \boldsymbol{E}_i\boldsymbol{E}_0^*, \boldsymbol{I}), \\
\mathscr{L}_{1,n-1}(\boldsymbol{E}_0^*\boldsymbol{E}_i - \boldsymbol{E}_i\boldsymbol{E}_0^*, \boldsymbol{I})\mathscr{L}_{1,n-1}(\boldsymbol{E}_0^*, \boldsymbol{I}) &= \mathscr{L}_{1,n-1}(c_i\boldsymbol{E}_0^* - \boldsymbol{E}_0^*\boldsymbol{E}_i, \boldsymbol{I}) + \mathscr{L}_{1,1,n-2}(\boldsymbol{E}_0^*,\boldsymbol{E}_0^*\boldsymbol{E}_i - \boldsymbol{E}_i\boldsymbol{E}_0^*, \boldsymbol{I}).
\end{align*}
By subtraction, since $\mathscr{L}_{1,n-1}(\boldsymbol{E}_0^*, \boldsymbol{I})$ is in $\boldsymbol{\mathcal{T}}$ from Lemma \ref{EI} and since $\mathscr{L}_{1,n-1}(\boldsymbol{E}_0^*\boldsymbol{E}_i - \boldsymbol{E}_i\boldsymbol{E}_0^*, \boldsymbol{I}) \in \boldsymbol{\mathcal{T}}$,
we obtain $\mathscr{L}_{1,n-1}(\boldsymbol{E}_0^*\boldsymbol{E}_i + \boldsymbol{E}_i\boldsymbol{E}_0^*, \boldsymbol{I}) \in \boldsymbol{\mathcal{T}}$.
Thus we get $\mathscr{L}_{1,n-1}(\boldsymbol{E}_0^*\boldsymbol{E}_i, \boldsymbol{I})$, $\mathscr{L}_{1,n-1}(\boldsymbol{E}_i\boldsymbol{E}_0^*, \boldsymbol{I})$ are in $\boldsymbol{\mathcal{T}}$.
Similar arguments apply
to $\mathscr{L}_{1,n-1}(\boldsymbol{E}_0\boldsymbol{E}_i^*, \boldsymbol{I})$, $\mathscr{L}_{1,n-1}(\boldsymbol{E}_i^*\boldsymbol{E}_0, \boldsymbol{I})$
and so they are also in $\boldsymbol{\mathcal{T}}$.

Next we observe that
\begin{align*}
\mathscr{L}_{1,n-1}(\boldsymbol{E}_i\boldsymbol{E}_0^*, \boldsymbol{I})\mathscr{L}_{1,n-1}(\boldsymbol{E}_0^*\boldsymbol{E}_i, \boldsymbol{I}) &= \mathscr{L}_{1,n-1}(\boldsymbol{E}_i\boldsymbol{E}_0^*\boldsymbol{E}_i, \boldsymbol{I}) + \mathscr{L}_{1,1,n-1}(\boldsymbol{E}_i\boldsymbol{E}_0^*, \boldsymbol{E}_0^*\boldsymbol{E}_i, \boldsymbol{I}),\\
\mathscr{L}_{1,n-1}(\boldsymbol{E}_0^*\boldsymbol{E}_i, \boldsymbol{I})\mathscr{L}_{1,n-1}(\boldsymbol{E}_i\boldsymbol{E}_0^*, \boldsymbol{I}) &= \mathscr{L}_{1,n-1}(c_i\boldsymbol{E}_0^*, \boldsymbol{I}) + \mathscr{L}_{1,1,n-1}(\boldsymbol{E}_i\boldsymbol{E}_0^*, \boldsymbol{E}_0^*\boldsymbol{E}_i, \boldsymbol{I})
\end{align*}
and then we obtain $\mathscr{L}_{1,n-1}(\boldsymbol{E}_i\boldsymbol{E}_0^*\boldsymbol{E}_i, \boldsymbol{I}) \in \boldsymbol{\mathcal{T}}$.
Similar arguments apply
to $\mathscr{L}_{1,n-1}(\boldsymbol{E}_i^*\boldsymbol{E}_0\boldsymbol{E}_i^*, \boldsymbol{I})$
and so it is also in $\boldsymbol{\mathcal{T}}$.
The result follows.
\end{proof}

Recall the matrices $\{\boldsymbol{F}_j\}_{j=0}^m$, $\{\boldsymbol{F}_j^*\}_{j=0}^m$ from \eqref{FF*}.
and recall the subalgebras $\mathcal{F}$, $\mathcal{F}^*$ generated by 
$\{\boldsymbol{F}_j\}_{j=0}^m$, $\{\boldsymbol{F}_j^*\}_{j=0}^m$, respectively.
By Lemma \ref{lemF} (i), $\mathcal{F} \neq \{\boldsymbol{O}\}$ and $\mathcal{F}^* \neq \{\boldsymbol{O}\}$.
By Lemma \ref{EEI} and the definitions of $\{\boldsymbol{F}_j\}_{j=0}^m$, $\{\boldsymbol{F}_j^*\}_{j=0}^m$,
we have $\mathscr{L}_{1,n-1}(\boldsymbol{F}_i, \boldsymbol{I})$, $\mathscr{L}_{1,n-1}(\boldsymbol{F}_i^*, \boldsymbol{I})$ are in the Terwilliger algebra of $\mathfrak{X}(m,n)$ for $0 \le i \le m$.
By Proposition \ref{sym:gen} and Lemma \ref{lem2} (v), (vi),
the $n$-fold symmetric tensor algebras $\operatorname{Sym}^n(\mathcal{F})$ and $\operatorname{Sym}^n(\mathcal{F}^*)$
are also in the Terwilliger algebra of $\mathfrak{X}(m,n)$.
For $\lambda \in \mathbb{I}(m,n)$, we define
\begin{align*}
\mathscr{L}_\lambda(\boldsymbol{F}_\bullet) = \mathscr{L}_\lambda(\boldsymbol{F}_0, \boldsymbol{F}_1, \ldots, \boldsymbol{F}_m),
&&
\mathscr{L}_\lambda(\boldsymbol{F}_\bullet^*) = \mathscr{L}_\lambda(\boldsymbol{F}_0^*, \boldsymbol{F}_1^*, \ldots, \boldsymbol{F}_m^*).
\end{align*}

\begin{prop}\label{T0(m,n)}
The following algebras are the same.
\begin{itemize}
\item The primary subalgebra of the Terwilliger algebra of $\mathfrak{X}(m,n)$.
\item The algebra generated by $\operatorname{Sym}^n(\mathcal{F})$ and $\operatorname{Sym}^n(\mathcal{F}^*)$.
\item The $n$-fold symmetric tensor algebra $\operatorname{Sym}^n(\mathcal{T}_0)$, where $\mathcal{T}_0$ is the primary subalgebra of the Terwilliger algebra of $\mathfrak{X}(m,1)$.
\end{itemize}
\end{prop}
\begin{proof}
Let $\boldsymbol{\mathcal{T}_0}$ denote the primary subalgebra of the Terwilliger algebra of $\mathfrak{X}(m,n)$.
Let $\boldsymbol{\mathcal{T}_0}'$ denote the algebra generated by $\operatorname{Sym}^n(\mathcal{F})$ and $\operatorname{Sym}^n(\mathcal{F}^*)$.
It suffices to show (i) $\boldsymbol{\mathcal{T}_0} \subset \boldsymbol{\mathcal{T}_0}'$, (ii) $\boldsymbol{\mathcal{T}_0}' \subset \operatorname{Sym}^n(\mathcal{T}_0)$ and (iii) $\dim \boldsymbol{\mathcal{T}_0} = \dim (\operatorname{Sym}^n(\mathcal{T}_0))$.
\begin{enumerate}
\item
The primary subalgebra $\mathcal{T}_0'$ has a standard basis
$\mathscr{L}_\lambda(\boldsymbol{E}_\bullet^*) \left(\boldsymbol{E}_0^{\otimes n}\right) \mathscr{L}_\mu(\boldsymbol{E}_\bullet^*)$ for
$\lambda, \mu \in \mathbb{I}(m,n)$.
By Lemma \ref{lem2} (vi), $\boldsymbol{E}_i^*\boldsymbol{E}_0\boldsymbol{E}_j^* = \boldsymbol{F}_i^*\boldsymbol{F}_0\boldsymbol{F}_j^*$ for $0 \le i,j \le m$.
This implies 
\[
\mathscr{L}_\lambda(\boldsymbol{E}_\bullet^*) \left(\boldsymbol{E}_0^{\otimes n}\right) \mathscr{L}_\mu(\boldsymbol{E}_\bullet^*) = \mathscr{L}_\lambda(\boldsymbol{F}_\bullet^*) \left(\boldsymbol{F}_0^{\otimes n}\right) \mathscr{L}_\mu(\boldsymbol{F}_\bullet^*) \in \boldsymbol{\mathcal{T}_0}'
\]
since every components are of the form $\boldsymbol{E}_i^*\boldsymbol{E}_0\boldsymbol{E}_j^* = \boldsymbol{F}_i^*\boldsymbol{F}_0\boldsymbol{F}_j^*$.
Thus $\boldsymbol{\mathcal{T}_0} \subset \boldsymbol{\mathcal{T}_0}'$.

\item By Lemma \ref{lemF} (iv),
$\mathcal{T}_0$ is generated by $\mathcal{F}$ and $\mathcal{F}^*$.
Thus $\boldsymbol{\mathcal{T}_0}' \subset \operatorname{Sym}^n(\mathcal{T}_0)$.

\item 
$\dim \boldsymbol{\mathcal{T}_0}$ is the number of relations $\binom{m+n}{n}$ from \eqref{relations}.
On the other hands, $\dim (\operatorname{Sym}^n(\mathcal{T}_0)) = \binom{\dim \mathcal{T}_0 + n - 1}{n}$ by Lemma \ref{sym}.
By Proposition \ref{T0}, the primary subalgebra $\mathcal{T}_0$ has dimension $m+1$.
Thus both have the same dimension.
\end{enumerate}
\end{proof}

\begin{prop}\label{vm=1,q=2}
The Terwilliger algebra of $\mathfrak{X}(m,n)$ coincides with the primary subalgebra if $m=1$ and $q_1=2$.
\end{prop}
\begin{proof}
By Lemma \ref{m=1,q=2},
$\mathscr{L}_\lambda(\boldsymbol{E}_\bullet) = \mathscr{L}_\lambda(\boldsymbol{F}_\bullet)$
and
$\mathscr{L}_\lambda(\boldsymbol{E}_\bullet^*) = \mathscr{L}_\lambda(\boldsymbol{F}_\bullet^*)$
for $\lambda \in \mathbb{I}(m,n)$.
By Propositions \ref{BM} and \ref{DBM},
$\operatorname{Sym}^n(\mathcal{F})$ and $\operatorname{Sym}^n(\mathcal{F}^*)$
are the Bose-Mesner algebra and dual Bose-Mesner algebra of $\mathfrak{X}(m,n)$, respectively.
In particular, the algebra generated by $\operatorname{Sym}^n(\mathcal{F})$ and $\operatorname{Sym}^n(\mathcal{F}^*)$ is the Terwilliger algebra of $\mathfrak{X}(m,n)$.
By Proposition \ref{T0(m,n)},
The Terwilliger algebra of $\mathfrak{X}(m,n)$ coincides with the primary subalgebra.
\end{proof}

For the rest of this section, unless otherwise stated we assume $m \ge 2$ or $q_1 \ge 3$.
Recall the matrices $\{\boldsymbol{G}_j\}_{j=1}^m$, $\{\boldsymbol{G}_j^*\}_{j=1}^m$ from \eqref{GG*}
and recall that
the subalgebras $\mathcal{G}$, $\mathcal{G}^*$ generated by $\{\boldsymbol{G}_j\}_{j=1}^m$, $\{\boldsymbol{G}_j^*\}_{j=1}^m$, respectively.
By Lemma \ref{T1(m,1)}, $\mathcal{G} \neq \{\boldsymbol{O}\}$ and $\mathcal{G}^* \neq \{\boldsymbol{O}\}$.
By Lemmas \ref{EI}, \ref{EEI} and the definitions of $\{\boldsymbol{G}_j\}_{j=1}^m$, $\{\boldsymbol{G}_j^*\}_{j=1}^m$,
we have $\mathscr{L}_{1,n-1}(\boldsymbol{G}_i, \boldsymbol{I})$, $\mathscr{L}_{1,n-1}(\boldsymbol{G}_i^*, \boldsymbol{I})$ are in the Terwilliger algebra of $\mathfrak{X}(m,n)$ for $1 \le i \le m$.
By Proposition \ref{sym:gen} and Lemma \ref{lem3} (i), (ii),
the $n$-fold symmetric tensor algebras $\operatorname{Sym}^n(\mathcal{G})$ and $\operatorname{Sym}^n(\mathcal{G}^*)$
are also in the Terwilliger algebra of $\mathfrak{X}(m,n)$.
For $\mu \in \mathbb{I}(m-1,n)$, we define
\begin{align*}
\mathscr{L}_\mu(\boldsymbol{G}_\bullet) = \mathscr{L}_\mu(\boldsymbol{G}_1, \boldsymbol{G}_2, \ldots, \boldsymbol{G}_m),
&&
\mathscr{L}_\mu(\boldsymbol{G}_\bullet^*) = \mathscr{L}_\mu(\boldsymbol{G}_1^*, \boldsymbol{G}_2^*, \ldots, \boldsymbol{G}_m^*).
\end{align*}
Note that
by Lemma \ref{lemG} (v), $\mathcal{F} \cap \mathcal{G} = \mathcal{F}^* \cap \mathcal{G}^* = 0$.

\begin{lem}\label{symH}
Assume $m \ge 2$ or $q_1 \ge 3$.
$\operatorname{Sym}^n(\mathcal{F} \oplus \mathcal{G})$ and $\operatorname{Sym}^n(\mathcal{F}^* \oplus \mathcal{G}^*)$
are in the Terwilliger algebra of $\mathfrak{X}(m,n)$.
\end{lem}
\begin{proof}
We have already seen that
$\mathscr{L}_{1,n-1}(\boldsymbol{F}_i, \boldsymbol{I})$ $(0 \le i \le m)$ and $\mathscr{L}_{1,n-1}(\boldsymbol{G}_i, \boldsymbol{I})$ $(1 \le i \le m)$ are in the Terwilliger algebra of $\mathfrak{X}(m,n)$.
By Proposition \ref{sym:gen}, Lemma \ref{lem2} (v) and Lemma \ref{lem3} (i) (iii),
$\operatorname{Sym}^n(\mathcal{F} \oplus \mathcal{G})$ are also in the Terwilliger algebra of $\mathfrak{X}(m,n)$.
Similarly, the assertion for $\operatorname{Sym}^n(\mathcal{F}^* \oplus \mathcal{G}^*)$ holds.
\end{proof}

\begin{prop}\label{symHH*}
Assume $m \ge 2$ or $q_1 \ge 3$.
The following hold.
\begin{enumerate}
\item The Bose-Mesner algebra of $\mathfrak{X}(m,n)$ is in $\operatorname{Sym}^n(\mathcal{F} \oplus \mathcal{G})$.
\item The dual Bose-Mesner algebra of $\mathfrak{X}(m,n)$ is in $\operatorname{Sym}^n(\mathcal{F}^* \oplus \mathcal{G}^*)$.
\item The Terwilliger algebra of $\mathfrak{X}(m,n)$ is generated by $\operatorname{Sym}^n(\mathcal{F} \oplus \mathcal{G})$ and $\operatorname{Sym}^n(\mathcal{F}^* \oplus \mathcal{G}^*)$.
\end{enumerate}
\end{prop}
\begin{proof}
By Lemma \ref{lem2} (i) and the definitions of $\{\boldsymbol{G}_j\}_{j=1}^m$,
we have $\boldsymbol{E}_0 = \boldsymbol{F}_0$ and $\boldsymbol{E}_i = \boldsymbol{F}_i + \boldsymbol{G}_i$ $(1 \le i \le m)$.
Thus $\boldsymbol{E}_i \in \mathcal{F} \oplus \mathcal{G}$ for $0 \le i \le m$.
In particular, $\mathscr{L}_{\lambda}(\boldsymbol{E}_\bullet) \in \operatorname{Sym}^n(\mathcal{F} \oplus \mathcal{G})$ for  $\lambda \in \mathbb{I}(m,n)$.
This means the Bose-Mesner algebra of $\mathfrak{X}(m,n)$ is in $\operatorname{Sym}^n(\mathcal{F} \oplus \mathcal{G})$.
Similarly, the dual Bose-Mesner algebra of $\mathfrak{X}(m,n)$ is in $\operatorname{Sym}^n(\mathcal{F}^* \oplus \mathcal{G}^*)$.
Let $\mathcal{T}$ denote the Terwilliger algebra of $\mathfrak{X}(m,n)$
and let $\mathcal{T}'$ denote the algebra generated by $\operatorname{Sym}^n(\mathcal{F} \oplus \mathcal{G})$ and $\operatorname{Sym}^n(\mathcal{F}^* \oplus \mathcal{G}^*)$.
Since $\mathcal{T}$ is generated by the Bose-Mesner algebra and the dual Bose-Mesner algebra, 
$\mathcal{T} \subset \mathcal{T}'$.
The other containment follows from Lemma \ref{symH}.
\end{proof}

\begin{nota}\label{Td}
Assume $m \ge 2$ or $q_1 \ge 3$.
For non-negative integer $e,d$, let $\mathcal{T}_d(m,e+d)$ denote the algebra generated by 
$\operatorname{Sym}^{e}(\mathcal{F}) \odot \operatorname{Sym}^d(\mathcal{G})$
and
$\operatorname{Sym}^{e}(\mathcal{F}^*) \odot \operatorname{Sym}^d(\mathcal{G}^*)$.
\end{nota}
\begin{prop}\label{T=sumTd}
Referring to Notation \ref{Td},
the Terwillger algebra of $\mathfrak{X}(m,n)$ is $\bigoplus_{d=0}^n\mathcal{T}_d(m,n)$.
\end{prop}
\begin{proof}
Let $\mathcal{T}$ be the Terwillger algebra of $\mathfrak{X}(m,n)$
and let $\mathcal{T}'$ be the algebra generated by $\mathcal{T}_d(m,n)$ for $0 \le d \le n$.
By Lemma \ref{lemG} (v),
$\mathcal{T}'$ is a direct sum of $\mathcal{T}_d(m,n)$ for $0 \le d \le n$.
By Lemmas \ref{sym:V+W} and \ref{symH},
we have $\operatorname{Sym}^{n-d}(\mathcal{F}) \odot \operatorname{Sym}^d(\mathcal{G}) \subset \operatorname{Sym}^n(\mathcal{F} \oplus \mathcal{G}) \subset \mathcal{T}$.
Similarly, we have $\operatorname{Sym}^{n-d}(\mathcal{F}^*) \odot \operatorname{Sym}^d(\mathcal{G}^*) \subset \mathcal{T}$.
In particular, $\mathcal{T}'$ is a subalgebra of $\mathcal{T}$.
By Lemma \ref{sym:V+W},
$\operatorname{Sym}^n(\mathcal{F} \oplus \mathcal{G}) = 
\bigoplus_{d=0}^n
\operatorname{Sym}^{n-d}(\mathcal{F}) \odot \operatorname{Sym}^d(\mathcal{G})$ and 
$\operatorname{Sym}^n(\mathcal{F}^* \oplus \mathcal{G}^*) = 
\bigoplus_{d=0}^n
\operatorname{Sym}^{n-d}(\mathcal{F}^*) \odot \operatorname{Sym}^d(\mathcal{G}^*)$
are in $\mathcal{T}'$.
By Proposition \ref{symHH*} (iii),
$\mathcal{T}$ is contained in $\mathcal{T}'$.
Since $\mathcal{T}'$ is a subalgebra of $\mathcal{T}$, they must coincide.
The result follows.
\end{proof}

\begin{prop}\label{Td=T0Td}
Referring to Notation \ref{Td},
for $0 \le d \le n$, we have $\mathcal{T}_d(m,n) = \mathcal{T}_0(m,n-d) \odot \mathcal{T}_d(m,d)$.
Moreover, $\mathcal{T}_d(m,n)$ is isomorphic, as $\mathbb{C}$-algebras, to $\mathcal{T}_0(m,n-d) \otimes \mathcal{T}_d(m,d)$.
\end{prop}
\begin{proof}
We may assume $1 \le d \le n-1$ since the assertion is clear by $\mathcal{T}_0(m,0) = \mathbb{C}$ for $d=0$ or $d=n$.
Since $\operatorname{Sym}^{n-d}(\mathcal{F}) \subset \mathcal{T}_0(m,n-d)$
and $\operatorname{Sym}^d(\mathcal{G}) \subset \mathcal{T}_d(m,d)$,
we have 
$\operatorname{Sym}^{n-d}(\mathcal{F}) \odot \operatorname{Sym}^d(\mathcal{G}) \subset \mathcal{T}_0(m,n-d) \odot \mathcal{T}_d(m,d)$.
Similary, we have 
$\operatorname{Sym}^{n-d}(\mathcal{F}^*) \odot \operatorname{Sym}^d(\mathcal{G}^*) \subset \mathcal{T}_0(m,n-d) \odot \mathcal{T}_d(m,d)$.
Therefore, $\mathcal{T}_d(m,n) \subset \mathcal{T}_0(m,n-d) \odot \mathcal{T}_d(m,d)$.
It remains to show the other containment.
Recall from Lemma \ref{lemG} (iv), there are the common identity element 
$\boldsymbol{G}^\natural$ for $\mathcal{G}$ and $\mathcal{G}^*$.
Thus, $(\boldsymbol{G}^\natural)^{\otimes d}$ is a common identity element 
for $\operatorname{Sym}^d(\mathcal{G})$ and $\operatorname{Sym}^d(\mathcal{G}^*)$.
By Lemmas \ref{sym:CdotD} and \ref{lemG} (v),
we have $\mathcal{T}_0(m,n-d) \odot (\boldsymbol{G}^\natural)^{\otimes d} \subset \mathcal{T}_d(m,n)$.
Similarly, 
there exists the common identity element $(\boldsymbol{F}^\natural)^{\otimes (n-d)}$
for $\operatorname{Sym}^{n-d}(\mathcal{F})$ and $\operatorname{Sym}^{n-d}(\mathcal{F}^*)$.
Then we have $(\boldsymbol{F}^\natural)^{\otimes (n-d)} \odot \mathcal{T}_d(m,d) \subset \mathcal{T}_d(m,n)$.
Therefore, by Lemmas \ref{sym:CdotD} and \ref{lemG} (v), we obtain $\mathcal{T}_0(m,n-d) \odot \mathcal{T}_d(m,d) \subset \mathcal{T}_d(m,n)$.
The first assertion follows.
The second assertion follows from Lemma By Lemmas \ref{sym:CdotD} and \ref{lemG} (v).
\end{proof}

\begin{nota}\label{C}
Assume $m \ge 2$ or $q_1 \ge 3$.
Referring to  Notation \ref{Lambda}, for $\lambda = (\lambda_0, \lambda_1, \ldots, \lambda_{m-1}), \mu = (\mu_0, \mu_1, \ldots, \mu_{m-1}) \in \mathbb{I}(m-1,n)$, let $\Theta(\lambda,\mu)$ denote the set of $m$ by $m$ matrices $C = [c_{i,j}]_{i,j=1}^m$ whose entries are non-negative integers such that
\begin{itemize}
\item the entry $c_{i,j} = 0$ if $(i,j) \not\in \Lambda(m)$,
\item the $i$-th row sum is $\lambda_{i-1}$ for $1 \le i \le m$,
\item the $j$-th column sum is $\mu_{j-1}$ for $1 \le j \le m$.
\end{itemize}
\end{nota}

With this notation, we refer the classical result.
\begin{prop}[c.f.\ \cite{BR}]\label{Existence}
Referring to Notation \ref{C},
for $\lambda = (\lambda_0, \lambda_1, \ldots, \lambda_{m-1}), \mu = (\mu_0, \mu_1, \ldots, \mu_{m-1}) \in \mathbb{I}(m-1,n)$, the set $\Theta(\lambda,\mu)$ is non-empty if and only if 
\begin{align*}
\sum_{i \in \mathcal{I}} \lambda_{i-1} \le \sum_{j \not\in \mathcal{J}} \mu_{j-1},
&&
\sum_{i \not\in \mathcal{I}} \lambda_{i-1} \le \sum_{j \in \mathcal{J}} \mu_{j-1},
\end{align*}
for any $\mathcal{I}, \mathcal{J} \subset \{1, 2,\ldots, m\}$ such that $(\mathcal{I} \times \mathcal{J}) \cap \Lambda(m) = \emptyset$.
\end{prop}

\begin{lem}\label{nonempty}
Referring to Notation \ref{C},
there always exist $\lambda, \mu \in \mathbb{I}(m-1,n)$ such that
$\Theta(\lambda, \mu) \neq \emptyset$.
\end{lem}
\begin{proof}
Set
$\lambda = \mu = (0,0,\ldots,0,n) \in \mathbb{I}(m-1,n)$
Define $C = [c_{i,j}]_{i,j=1}^m$ by $c_{m,m}=n$ and $0$ otherwise.
Then it is easy to prove that $C \in \Theta(\lambda, \mu)$.
\end{proof}

\begin{lem}\label{Gnatural}
The common identity element for $\operatorname{Sym}^n(\mathcal{G})$ and $\operatorname{Sym}^n(\mathcal{G}^*)$
can be expressed as
\[
\sum_{\tau \in \mathbb{I}(m-1,n)}
\mathscr{L}_\tau(\boldsymbol{G}_\bullet)
=
\sum_{\tau \in \mathbb{I}(m-1,n)}
\mathscr{L}_\tau(\boldsymbol{G}_\bullet^*).
\]
\end{lem}
\begin{proof}
Routine using Lemma \ref{sym:basis}.
\end{proof}

\begin{lem}\label{commute2}
Referring to Notation \ref{C},
for $\lambda, \mu \in \mathbb{I}(m-1,n)$, the following hold.
\begin{enumerate}
\item If $(\lambda, \mu)$ satisfies the equivalent conditions in Proposition \ref{Existence},
\[
\mathscr{L}_\mu(\boldsymbol{G}_\bullet)\mathscr{L}_\lambda(\boldsymbol{G}_\bullet^*) =
\mathscr{L}_\lambda(\boldsymbol{G}_\bullet^*)\mathscr{L}_\mu(\boldsymbol{G}_\bullet) = 
\sum_{C \in \Theta(\lambda, \mu)} 
\mathscr{L}_C(\{\boldsymbol{G}_j\boldsymbol{G}_i^*\}_{i,j=1}^m).
\]
In particular, this common product is non-zero.
\item If $(\lambda, \mu)$ does not satisfy the equivalent conditions in Proposition \ref{Existence},
both
$\mathscr{L}_\mu(\boldsymbol{G}_\bullet)\mathscr{L}_\lambda(\boldsymbol{G}_\bullet^*)$ and 
$\mathscr{L}_\lambda(\boldsymbol{G}_\bullet^*)\mathscr{L}_\mu(\boldsymbol{G}_\bullet)$ are the zero matrix.
\end{enumerate}
\end{lem}
\begin{proof}
Routine using Lemma \ref{commute} and Notation \ref{C}.
\end{proof}

\begin{prop}\label{Tn(m,n)}
Assume $m \ge 2$ or $q_1 \ge 3$.
Let $\Omega(m,n)$ denote the set of pairs $(\lambda, \mu)$ with $\lambda, \mu \in \mathbb{I}(m-1,n)$ satisfying the equivalent conditions in Proposition \ref{Existence}.
Referring to Notation \ref{Td},
$\mathcal{T}_n(m,n)$ is a commutative algebra whose basis are
$\{\mathscr{L}_\mu(\boldsymbol{G}_\bullet)\mathscr{L}_\lambda(\boldsymbol{G}_\bullet^*) \mid (\lambda, \mu) \in \Omega(m,n)\}$.
\end{prop}

\begin{proof}
At first, observe that by Lemma \ref{nonempty}, $\Omega(m,n)$ is not empty.
Let $\mathcal{T}'$ denote the linear subspace of $\mathcal{T}_n(m,n)$ spanned by $\{\mathscr{L}_\mu(\boldsymbol{G}_\bullet)\mathscr{L}_\lambda(\boldsymbol{G}_\bullet^*) \mid (\lambda, \mu) \in \Omega(m,n)\}$.
Observe that these matrices are closed under multiplication by Lemmas \ref{commute} and \ref{commute2}.
Also the multiplication is commutative by Lemmas \ref{commute}, \ref{commute2}.
This means $\mathcal{T}'$ is a commutative subalgebra of $\mathcal{T}_n(m,n)$.
By Lemma \ref{Gnatural}, we have
\[
\mathscr{L}_\mu(\boldsymbol{G}_\bullet) = 
\sum_{\lambda \in \mathbb{I}(m-1,n)} 
\mathscr{L}_\mu(\boldsymbol{G}_\bullet)\mathscr{L}_\lambda(\boldsymbol{G}_\bullet)
=
\sum_{\lambda \in \mathbb{I}(m-1,n)} 
\mathscr{L}_\mu(\boldsymbol{G}_\bullet)\mathscr{L}_\lambda(\boldsymbol{G}_\bullet^*)
\in \mathcal{T}'.
\]
Similarly, $\mathscr{L}_\mu(\boldsymbol{G}_\bullet^*) \in \mathcal{T}'$.
This means $\mathcal{T}_n(m,n)$ is contained in $\mathcal{T}'$.
Since $\mathcal{T}'$ is a subalgebra of $\mathcal{T}_n(m,n)$, they must coincide.
The result follows.
\end{proof}

\begin{thm}[c.f.\ {\cite{G, LMP}}]\label{T(1,n)}
The Terwilliger algebra of $\mathfrak{X}(m,n)$ coincides with the primary subalgebra if and only if $m=1$ and $q_1=2$.
\end{thm}
\begin{proof}
If $m=1$ and $q_1=2$, then the result follows from Proposition \ref{vm=1,q=2}.
It remains to consider the case $m\ge 2$ or $q_1 \ge 3$.
By Proposition \ref{Tn(m,n)},
$\dim(\mathcal{T}_n(m,n)) = |\Omega(m,n)| \ge 1$.
By Proposition \ref{T=sumTd},
$\mathcal{T}_0(m,n)$ is a proper subset of 
the Terwilliger algebra of $\mathfrak{X}(m,n)$.
By Propositions \ref{T0(m,n)}, $\mathcal{T}_0(m,n)$ is the primary subalgebra of the Terwilliger algebra of $\mathfrak{X}(m,n)$.
This means it does not coincide with the primary subalgebra.
\end{proof}

\begin{thm}\label{main}Assume $m \ge 2$ or $q_1 \ge 3$.
The Terwilliger algebra of $\mathfrak{X}(m,n)$ is
\[
\bigoplus_{d=0}^n \operatorname{Sym}^{n-d}(\mathcal{T}_0) \odot \mathcal{T}_d(m,d),
\]
where $\mathcal{T}_0$ is the primary subalgebra of the Terwilliger algebra of $\mathfrak{X}(m,1)$
and $\mathcal{T}_d(m,d)$ is the commutative subalgebra of the Terwilliger algebra of $\mathfrak{X}(m,d)$
generated by $\operatorname{Sym}^d(\mathcal{G})$ and $\operatorname{Sym}^d(\mathcal{G}^*)$.
\end{thm}
\begin{proof}
Combine Propositions \ref{T=sumTd}, \ref{Td=T0Td}, \ref{T0(m,n)} and \ref{Tn(m,n)}.
\end{proof}

\begin{cor}\label{cor:main}
The Terwilliger algebra of $\mathfrak{X}(m,n)$ is isomorphic, as $\mathbb{C}$-algebra, to
\[
\begin{cases}
\bigoplus_{d=0}^n
\operatorname{Mat}_{n-d+1}(\mathbb{C})
& \text{if $m = 1$ and $q_1 = 2$,}\\
\bigoplus_{d=0}^n \left(\operatorname{Mat}_{\binom{m+n-d}{n-d}}(\mathbb{C})\right) \otimes \mathbb{C}^{|\Omega(m,n)|}
&\text{if $m \ge 2$ or $q_1 \ge 3$,}
\end{cases}
\]
where $\Omega(m,n)$ is the set of pairs $(\lambda, \mu)$ with $\lambda, \mu \in \mathbb{I}(m-1,n)$ satisfying the equivalent conditions in Proposition \ref{Existence}.
\end{cor}
\begin{proof}
Routine using Proposition \ref{Tn(m,n)} and Theorems \ref{T(1,n)}, \ref{main}.
\end{proof}

\section{Spacial cases}
In this section, we simplify Theorem \ref{main} by restating them for small parameters.

\begin{lem}\label{|Theta|}
Referring to Notation \ref{C},
we have
\[
\left|\bigsqcup_{\lambda,\mu \in \mathbb{I}(m-1,n)} \Theta(\lambda,\mu)\right| = \binom{|\Lambda(m)|+n-1}{n}.
\]
\end{lem}
\begin{proof}
Set $\Theta = \bigsqcup_{\lambda,\mu \in \mathbb{I}(m-1,n)} \Theta(\lambda,\mu)$.
By Notation \ref{C}, $\Theta$ is the set of $m$ by $m$ matrices $C = [c_{i,j}]$ $(1 \le i,j \le m)$ whose entries are non-negative integers such that the entry $c_{i,j} = 0$ if $(i,j) \not\in \Lambda(m)$ and that the sum of all entries is $n$.
Therefore it suffices to count the number of solutions to the equation:
\begin{align*}
\sum_{(i,j) \in \Lambda(m)}c_{i,j} = n,
&&
\text{each $c_{i,j}$ is a non-negative integer}.
\end{align*}
It is well-known that this number is equal to the number of multisets of cardinality $n$ taken from a set of size $|\lambda(m)|$, which is 
$\binom{|\Lambda(m)|+n-1}{n}$.
\end{proof}

\begin{lem}\label{|Omega|}
Assume $m \ge 2$ or $q_1 \ge 3$.
Let $\Omega(m,n)$ denote the set of pairs $(\lambda, \mu)$ with $\lambda, \mu \in \mathbb{I}(m-1,n)$ satisfying the equivalent conditions in Proposition \ref{Existence}.
Then we have
\[
|\Omega(m,n)| = \binom{|\Lambda(m)|+n-1}{n}
\]
if and only if one of the following conditions hold.
\begin{enumerate}
\item $m=1$ and $q_2 \ge 3$,
\item $m=2$,
\item $m=3$ and $q_2 = 2$,
\item $n=1$.
\end{enumerate}
\end{lem}
\begin{proof}
By Lemma \ref{|Theta|}, the condition $|\Omega(m,n)| = \binom{|\Lambda(m)|+n-1}{n}$ is equivalent to
\[
\left|\bigsqcup_{\lambda,\mu \in \mathbb{I}(m-1,n)} \Theta(\lambda,\mu)\right| = |\Omega(m,n)|.
\]
By the definition of $\Omega(m,n)$,
it is also equivalent to $|\Theta(\lambda,\mu)| = 1$ for $(\lambda,\mu) \in \Omega(m,n)$.

Suppose one of the four conditions (i), (ii), (iii), (iv) holds.
We show that $|\Theta(\lambda,\mu)| = 1$ for $(\lambda,\mu) \in \Omega(m,n)$.
Fix $(\lambda,\mu) \in \Omega(m,n)$.
If $m=1$ and $q_2 \ge 3$, then $C = [c_{1,1}] \in \Theta(\lambda,\mu)$ satisfies
$c_{1,1}= \lambda_0 = \mu_0$.
Thus there is only one such matrix $C$.
If $m=2$,
then $C = [c_{i,j}]_{i,j=1}^2 \in \Theta(\lambda,\mu)$ satisfies
\begin{align*}
c_{1,1}=0, 
&&
c_{1,2}=\lambda_0,
&&
c_{2,1}=\mu_0,
&&
c_{2,2}=\mu_1-\lambda_0=\lambda_1-\mu_0.
\end{align*}
Thus there is only one such matrix $C$.
Similarly, if $m=3$ and $q_2=2$, one can show there is only one such matrix $C \in \Theta(\lambda,\mu)$.
If $n=1$, then $C = [c_{i,j}]_{i,j=1}^m \in \Theta(\lambda,\mu)$ satisfies
$\sum_{i,j=1}^m c_{i,j} = 1$.
Since the entries of $C$ are non-negative integer,
there is one $1$ and $0$ elsewhere.
To satisfy the row and column sums condition, there is only one such matrix.
Hence, the result follows.

Suppose $|\Theta(\lambda,\mu)| = 1$ for $(\lambda,\mu) \in \Omega(m,n)$.
We also assume $m \ge 3$ and $n \ge 2$.
It suffices to show that $m=3$ and $q_2 = 3$.
Set $\lambda = \mu = (0,0,\ldots,0,1,n-1) \in \mathbb{I}(m-1,n)$.
Define
$C = [c_{i,j}]_{i,j=1}^m$, $C' = [c'_{i,j}]_{i,j=1}^m$ by 
\begin{align*}
c_{i,j} = \begin{cases}1 & \text{if $i=j=m-1$}, \\ n-1 & \text{if $i=j=m$}, \\ 0 & \text{otherwise}.\end{cases}
&&
c'_{i,j} = \begin{cases}1 & \text{if $i=m-1$, $j=m$}, \\ 1 & \text{if $i=m$, $j=m-1$},\\ n-2 & \text{if $i=j=m$}, \\ 0 & \text{otherwise}.\end{cases}
\end{align*}
One can show that $C$ and $C'$ have the same $i$-th row sum for $1 \le i \le m$  and have the same $j$-th column sum for $1 \le j \le m$.
Moreover, since $n \ge 2$, the entries are non-negative integer.
Since $m \ge 3$, $(m,m), (m-1,m), (m,m-1) \in \Lambda(m)$.
If $(m-1,m-1) \in \Lambda(m)$, then $C, C' \in \Theta(\lambda,\mu)$ for some $\lambda, \mu$, which is impossible since $C \neq C'$.
Thus we must have $(m-1,m-1) \not\in \Lambda(m)$.
By the definition of $\Lambda(m)$, we get $m=3$ and $q_2 = 3$.
So the result follows.
\end{proof}

\begin{prop}\label{Tn(m,n)'}
Let $m \ge 2$ or $q_1 \ge 3$.
Assume one of the four conditions (i), (ii), (iii), (iv) in Lemma \ref{|Omega|} holds.
Referring to Notation \ref{Td},
$\mathcal{T}_n(m,n)$ is the $n$-fold symmetric tensor algebra $\operatorname{Sym}^n(\mathcal{T}_1)$, where $\mathcal{T}_1$ is the subalgebra of the Terwilliger algebra of $\mathfrak{X}(m,1)$
generated by $\mathcal{G}$ and $\mathcal{G}^*$.
\end{prop}

\begin{proof}
Clearly, 
$\mathcal{T}_n(m,n)$ is contained in $\operatorname{Sym}^n(\mathcal{T}_1)$.
To show they are the same,
it suffices to prove they have the same dimension.
The dimension of $\mathcal{T}_n(m,n)$ is $|\Omega(m,n)|$ by Theorem \ref{Tn(m,n)}
and the dimension of $\operatorname{Sym}^n(\mathcal{T}_1)$ is 
$\binom{|\Lambda(m)|+n-1}{n}$ by Lemma \ref{sym} and Proposition \ref{T1(m,1)}.
They are the same by Lemma \ref{|Omega|}.
\end{proof}

\begin{thm}\label{main2}
Assume that one of the following holds:
(i) $m=1$, 
(ii) $m=2$,
(iii) $m=3$ and $q_2=2$,
(iv) $n=1$.
Then the Terwilliger algebra of $\mathfrak{X}(m,n)$ is the $n$-fold symmetric tensor algebra $\operatorname{Sym}^n(\mathcal{T})$,
where $\mathcal{T}$ is the Terwilliger algebra of $\mathfrak{X}(m,1)$.
\end{thm}
\begin{proof}
For the case $m=1$ and $q_1 = 2$, the result follows from Propositions \ref{prop}, \ref{T1(m,1)}, \ref{T0(m,n)} and Theorem \ref{T(1,n)}.

Now we assume 
$m \ge 2$ or $q_1 \ge 3$
and assume one of the four conditions (i), (ii), (iii), (iv) in Lemma \ref{|Omega|} holds.
By Theorem \ref{main} and Proposition \ref{Tn(m,n)'},
the Terwilliger algebra of $\mathfrak{X}(m,n)$ is
\[
\bigoplus_{d=0}^n \operatorname{Sym}^{n-d}(\mathcal{T}_0) \odot \operatorname{Sym}^d(\mathcal{T}_1),
\]
where $\mathcal{T}_0$ is the primary subalgebra of the Terwilliger algebra of $\mathfrak{X}(m,1)$ 
and $\mathcal{T}_1$ is the subalgebra of the Terwilliger algebra of $\mathfrak{X}(m,1)$
generated by $\mathcal{G}$ and $\mathcal{G}^*$.
By Proposition \ref{prop} and Lemma \ref{sym:V+W}, it is equal to
$\operatorname{Sym}^n(\mathcal{T})$, where $\mathcal{T}$ is the Terwilliger algebra of $\mathfrak{X}(m,1)$.
The result follows.
\end{proof}

\begin{cor}\label{cor:main2}
Assume that one of the following holds:
(i) $m=1$, 
(ii) $m=2$,
(iii) $m=3$ and $q_2=2$,
(iv) $n=1$.
The Terwilliger algebra of $\mathfrak{X}(m,n)$ is isomorphic, as $\mathbb{C}$-algebra, to
\[
\begin{cases}
\bigoplus_{d=0}^n\operatorname{Mat}_{n-d+1}(\mathbb{C})
& \text{if $m = 1$ and $q_1 = 2$,}\\
\bigoplus_{d=0}^n \left(\operatorname{Mat}_{\binom{m+n-d}{n-d}}(\mathbb{C})\right) \otimes \mathbb{C}^{\binom{m(m-1)/2+\varepsilon+d-1}{d}}
&\text{if $m \ge 2$ or $q_1 \ge 3$,}
\end{cases}
\]
where $\varepsilon = |\{i \mid 1 \le i \le m, q_i \ge 3\}|$.
\end{cor}

\begin{proof}
Routine using Corollary \ref{cor:main} and Lemma \ref{|Omega|}
and $|\Lambda(m)| = m(m-1)/2 + \varepsilon$.
\end{proof}

We point out that each of Theorem \ref{main2} and Corollary \ref{cor:main2} covers the case of $m=1$, corresponding to the ordinary Hamming scheme.

\section*{Acknowledgement}
The author wishes to express his thanks to William J.\ Martin for drawing the author’s attention to the Terwilliger algebra of ordered Hamming scheme.

\bigskip

\noindent
Yuta Watanabe \\
Department of Mathematics Education \\
Aichi University of Education \\
1 Hirosawa, Igaya-cho, Kariya, Aichi 448-8542, Japan. \\
email: \texttt{ywatanabe@auecc.aichi-edu.ac.jp}

\end{document}